\documentclass[11pt]{amsart}
\usepackage{amscd}
\usepackage{comment}
\usepackage{amsmath, amssymb, amsgen, amsthm, amscd, xspace, color}
\newcommand{\pf}{\begin{proof}}
\newcommand{\epf}{\end{proof}}
\newcommand{\eq}{\begin{equation}}
\newcommand{\eeq}{\end{equation}}
\newcommand{\eqn}{\begin{equation*}}
\newcommand{\eeqn}{\end{equation*}}

\newtheorem{theorem}[equation]{Theorem}

\newtheorem{prop}[equation]{Proposition}
\newtheorem{lemma}[equation]{Lemma}

\theoremstyle{remark}

\theoremstyle{definition}
\newtheorem{definition}[equation]{Definition}

\numberwithin{equation}{section}
\allowdisplaybreaks[1]
\begin{document}

\title{The spherical transform of a Schwartz function on the free two step nilpotent lie group}
\author{Jingzhe Xu}
\address[Xu]{Department of Mathematics, Hong Kong University of Science
and technology,
Clear Water Bay, Kowloon, Hong Kong SAR, China}
\email{jxuad@connect.ust.hk}
\abstract{Let $F(n)$ be a connected and simply connected free 2-step
nilpotent lie group and $K$ be a compact subgroup of Aut($F(n)$). We say
that $(K,F(n))$ is a Gelfand pair when the set of integrable $K$-invariant
functions on $F(n)$ forms an abelian algebra under convolution. In this
paper, we consider the case when $K=O(n)$. In
this case, the Gelfand space $(O(n),F(n)$ is equipped with the Godement-Plancherel
measure, and the spherical transform $\land:L_{O(n)}^{2}(F(n))\rightarrow L^{2}(\Delta (O(n),F(n)))$ is an isometry.
I will prove the Gelfand space $\Delta (O(n),F(n))$ is equipped with the Godement-Plancherel measure and the inversion formula. Both of which have something related to its correspond Heisenberg group. The main result in this paper provides a complete characterization of the set
$\varphi_{O(n)} (F(n))^{\wedge}$=$\{\widehat{f}\mid f\in \varphi_{O(n)} (F(n))\}$ of spherical transforms of $O(n)$-invariant Schwartz functions on $F(n)$. I show that a function $F$ on $\Delta (O(n),F(n))$ belongs to $\varphi_{O(n)} (F(n))^{\wedge}$ if and only if the functions obtained from $F$ via application of certain derivatives and difference operators satisfy decay conditions. }
\endabstract

%
\maketitle     
%
\section{Introduction}
What can one say about the spherical transform of a $O(n)$-invariant Schwartz function on $F(n)$? More precisely, letting $\varphi_{O(n)}(F(n))$ denote the
space of $O(n)$-invariant Schwartz functions on $F(n)$ we seek to characterize the subspace $\varphi_{O(n)} (F(n))^{\wedge}$=$\{\widehat{f}\mid f\in \varphi_{O(n)} (F(n))\}$ of $C_{0}(\Delta (O(n),F(n)))$, where $O(n)$-spherical transform $\widehat{f}$ $\rightarrow$ $C$ for a function $f\in L_{O(n)}^{1}(F(n))$ is defined by $\widehat{f}(\psi)=\int_{F(n)}f(x)\overline{\psi(x)}dx$.
 The main result in this paper in section 4 below, which provides a complete solution to this problem. Before describing the contents of this problem I wish to provide some background and motivation for the study of $\varphi_{O(n)} (F(n))$ via the spherical transform.

Schwartz functions have played an important role in harmonic analysis with nilpotent groups since the work of Kirilov [1]. Let $N$ be connected and simply connected nilpotent lie group with lie algebra $n$. The exponential map: $n\rightarrow N$ is a polynomial diffeomorphism and one defines the (Frechet) space $\varphi(N)$ of Schwartz functions on $N$ via identification with the usual space $\varphi(n)$ of Schwartz function on the vector space $n$:
$\varphi(N):=\{f:N\rightarrow C \mid f\circ exp \in \varphi(n)\}$. $\varphi(N)$ is dense in $L^{p}(N)$ for each $p$ and carries an algebra structure given by the convolution product. Moreover, it is known that the primitive ideal space for $\varphi(N)$ is isomorphic to that of both $L^{1}(N)$ and $C^{*}(N)$[2]. The Heisenberg groups $H_{a}$ are the simplest groups for which $\varphi(N)$ is non-abelian. Recall that the group Fourier transform for a function $f\in L^{1}(N)$ associates to $\pi\in \widehat{N}$, an irreducible unitary representation of $N$, the bounded operator $\pi(f)=\int_{N}f(x)\pi^{*}(x)dx$ in the representation space of $\pi$. This generates the usual Euclidean Fourier transform for the case $N=R^{n}$. The importance of Schwartz functions in Euclidean harmonic analysis arises from the fact that $\varphi(R^{n})$ is preserved by the Fourier transform. It is thus very natural to seek a characterization of $\varphi(N)$ via the group Fourier transform; a problem solved by Roger.Howe in [3].

One can sometimes obtain subalgebras of $\varphi(N)$ by considering "radial" functions. This is of interest even when $N=R^{n}$. Indeed, the algebra $\varphi_{O(n)}(R^{n})$ of radial Schwartz function on $R^{n}$ can be identified with $\varphi(R^{+})$ and the Fourier transform becomes a Hankel transform on $\varphi(R^{+})$. This is the spherical transform for the Gelfand pair obtained from the action of the orthogonal group $O(n)$ on $R^{n}$[4].

Theorem 4.4 provides conditions that are both necessary and sufficient for a function $F$ on $\Delta (O(n),F(n))$ to belong to the space $\varphi_{O(n)}(F(n))^{\wedge}$:

1 $F$ is continuous on $\Delta (O(n),F(n))$.

2 The function $F_{0}$ on $R$ defined by $F_{0}(r)=F(\phi^{r})$ belongs to $\phi(R)$.

3 The map $\lambda\rightarrow F(\phi^{r,\alpha,\lambda})$ is smooth on $R^{\times}$ and the functions $\partial_{\lambda}^{m}F(\phi^{r,\alpha,\lambda})$ satisfy certain decay conditions. In particular, $\partial_{\lambda}^{m}F(\phi^{r,\alpha,\lambda})$ is a rapidly decreasing sequence in $\alpha$ for each fixed $r\in R$ and $\lambda\in R^{\times}$.

4 Certain "derivatives" of $F$ also satisfy the three conditions above. These are defined on $\Delta_{1} (O(n),F(n))$ as specific combinations of $\partial_{\lambda}$ and "difference operators" which play the role of differentiation in the discrete parameter $\alpha\in \wedge$.

The precise formulation of these conditions can be found in Definition 3.2.1. The "derivatives" of functions in $\varphi_{O(n)}(F(n))^{\wedge}$ referred to above are operators corresponding to multiplication of functions in $\varphi_{O(n)}(F(n))$ by certain polynomials. The difference operators in the discrete parameter $\alpha\in \wedge$ are linear operators whose coefficients are "generalized binomial coefficients". These coefficients were introduced by Z.Yan in [5]. A summary of their properties is given below in Section 3.

I will prove the inversion formula for $(O(n),F(n))$ is of the form:

$f(x)=\frac{c}{(2\pi)^{2n+2}}\int_{R}\int_{R^{\times}}\sum_{\alpha\in \wedge}(dimP_{\alpha})\widehat{f}(\phi^{r,\alpha,\lambda})\phi^{r,\alpha,\lambda}(x)\left |\lambda \right |^{n}d\lambda dr$,
where $x=exp(X+A)\in F(n)$, $c$ is a fixed constant, $\phi^{r,\alpha,\lambda}(x)$ is the "type 1" $O(n)$-bounded spherical functions. Also, I will show the Godement-Plancherel measure $d\mu$ on $\Delta (O(n),F(n))$ is given by:

$\int_{\Delta (O(n),F(n))}F(\psi)d\mu(\psi)=\frac{c}{(2\pi)^{2n+2}}\int_{R}\int_{R^{\times}}\sum_{\alpha\in \wedge}(dimP_{\alpha})F(\phi^{r,\alpha,\lambda})\left |\lambda \right |^{n}d\lambda dr$.
Both of them are related to the corresponding formula of the Heisenberg group.

One consequence of the estimates involved in our characterization of $\varphi_{O(n)}(F(n))^{\wedge}$ is that $f\in \varphi_{O(n)}(F(n))$ can be recovered from $F=\widehat{f}$ via the inversion of the spherical transform.

%
\section{Notation and Preliminaries}
Let $G$ be a connected Lie group, $K$ a compact subgroup. Let $\pi$ denote the natural mapping of $G$ onto $X=G/K$ and as usual we put $o=\pi(e)$ and $\widetilde{f}=f\circ \pi$ if $f$ is any function on $X$. Let $D(G)$ denote the set of all left invariant differential operators on $G$, $D_{K}(G)$ the subspace of those which are also right invariant under $K$ and $D(G/K)$ the algebra of differential operators on $G/K$ invariant under all the translations
$\tau(g):xK\rightarrow gxK$ of $G/K$.

\begin{definition}
Let $\phi$ be a complex-valued function on $G/K$ of class $C^{\infty}$ which satisfies $\phi(\pi(e))=1$;$\phi$ is called a spherical function if

(i)$\phi^{\tau(k)}=\phi$ for all $k\in K$,

(ii)$D\phi=\lambda_{D}\phi$ for each $D\in D(G/K)$,

where $\lambda_{D}$ is a complex number.
\end{definition}
It is sometimes convenient to consider the function $\widetilde{\phi}=\phi\circ \pi$ on $G$ instead of $\phi$. We say that $\widetilde{\phi}$ is a spherical function on $G$ if and only if $\phi$ is a spherical function on $G/K$. Then a spherical function $\widetilde{\phi}$ on $G$ is characterized by being an eigenfunction of each operator in $D_{K}(G)$ and in addition satisfying the relations $\widetilde{\phi}(e)=1$,$\widetilde{\phi}(kgk^{'})=\widetilde{\phi}(g)$ for all $g\in G$ and all $k,k^{'}\in K$. The last condition will be called bi-invariance under $K$.

\begin{theorem}\label{equal}
Let $f$ be a complex-valued continuous function on $G$, not identically 0. Then $f$ is a spherical function if and only if

$\int_{K}f(xky)dk=f(x)f(y)$
for all $x,y\in G$.
\end{theorem}

Next, we consider the Heisenberg groups.
There are many ways to define Heisenberg group. I introduce two of them here.
[6]The first one is as follows:

Let $X$ be an arbitrary real vector space. Denote $X^{*}$ the dual space of $X$. Let $\varphi(X)$ be the Schwartz space of $X$, that is, the space of smooth, rapidly decreasing functions on $X$.(Unless specified otherwise, functions are complex-valued.)Let $T\subseteq C$ be the unit circle. Write
$e(t)=e^{2\pi it}$ for the usual exponential map from $R$ to $T$. Define operators on $\varphi(X)$ as follows.

(a)$\rho(x^{'})f(x)=f(x-x^{'})$ for $x,x^{'}\in X$  and $f\in \varphi(X)$

(b)$\rho(\xi)f(x)=e(\xi(x))f(x)$ for $\xi \in X^{*}$, and $x,f$ as in (a)

(c)$\rho(z)f(x)=zf(x)$ for $z\in T$,and $x,f$ as in (a)

These operators fit together to form a nice group of operators. Precisely, put

$H=X\times X^{*}\times T$.

Define a law of composition on $H$ by

$(x_{1},\xi_{1},z_{1})(x_{2},\xi_{2},z_{2})=(x_{1}+x_{2},\xi_{1}+\xi_{2},z_{1}z_{2}e(\xi_{1}(x_{2})))$.

It is easily verified that the law in above makes $H$ into a two-step nilpotent lie group with center $T$. We call $H$ the Heisenberg group. A straightforward computation shows that

$\rho:(x,\xi,z)\rightarrow \rho(x)\rho(\xi)\rho(z)$

defines an isomorphism of $H$ to a group of operators on $\varphi(X)$, or in other words, is a representation of $H$ on $\varphi(X)$.

As in usual, we let $L^{2}(X)$ be the Hilbert space of functions on $X$ which are square integrable with respect to Lebesgue measure. We know
$\varphi(X)$ is a dense subspace of $L^{2}(X)$, and in particular inherits the inner product from $L^{2}(X)$. One checks easily that the operators
$\rho(h)$ for $h\in H$ are isometries with respect to this inner product, so that in fact $\rho$ is the restriction to $\varphi(X)$ of a unitary representation, also denoted $\rho$, of $H$ on $L^{2}(X)$. As we will see, $\varphi(X)$ is intrinsically defined in terms if $\rho$.

For the second definition, we identity $H_{a}$ with $C^{a}\times R$ with multiplication given by

$(z,t)(z^{'},t^{'})=(z+z^{'},t+t^{'}+\frac{1}{2}\omega(z,z^{'}))$,
where $\omega(z,z^{'}):=-Im<z,z^{'}>=-Im(z.\overline{z}^{'})$ for $z,z^{'}\in C^{n}$, and $t,t^{'}\in R$.

It will occasionally be more congenial to have a "coordinate free" model for $H_{a}$. In this cases, we assume that $V$ is an $a$-dimensional vector space over $C$ equipped with a Hermitian inner product $<.,. >$. $H_{a}$ is then identified with $V\times R$, and the multiplication is given by

$(v,t)(v^{'},t^{'})=(v+v^{'},t+t^{'}+\frac{1}{2}\omega(v,v^{'}))$ where $\omega(v,v^{'}):=-Im<v,v^{'}>$. We will write $H_{V}$ for the Heisenberg group given by $(V,<.,. >)$.

The left-invariant vector fields generated by the one-parameter subgroups through $((0,\cdots ,0,1\pm i,0\cdots ,0),0)$ are written explicitly as

$Z_{j}=2\frac{\partial}{\partial \overline{Z_{j}}}+i\frac{Z_{j}}{2}\frac{\partial}{\partial t}$,
$\overline{Z_{j}}=2\frac{\partial}{\partial Z_{j}}-i\frac{\overline{Z_{j}}}{2}\frac{\partial}{\partial t}$.
In addition, let $T:=\frac{\partial}{\partial t}$, so that $\{Z_{1},\cdots ,Z_{a},\overline{Z_{1}},\cdots ,\overline{Z_{a}},T \}$ is a basis for the
lie algebra $n_{a}$ of $H_{a}$. With these notations one has [$Z_{j}$,$\overline{Z_{j}}$]=$-2iT$.

Consider a unimodular group $G$ with $K\subseteq G$ a compact subgroup. We denote the $L^{1}$-functions that are invariant under both the left and right actions of $K$ on $G$ by $L^{1}(G//K)$. These form a subalgebra of the group algebra $L^{1}(G)$ with respect to the convolution product
$f\ast g(x)=\int_{G}f(y)g(y^{-1}x)dy=\int_{G}f(xy^{-1})g(y)dy$.According to the traditional definition, one says that $K\subseteq G$ is a Gelfand pair if $L^{1}(G//K)$ is commutative.

Suppose now $K$ is a compact group acting on $N$, where $N$ is a connect and simply connected solvable lie group. By automorphism via some homomorphism $\phi:K\rightarrow Aut(N)$, one can form the semidirect product $K\propto N$, with group law

$(k_{1},x_{1})(k_{2},x_{2})=(k_{1}k_{2},x_{1}k_{1}.x_{2})$,
where we write $k.x$ for $\phi(k)(x)$. Right $K$-invariance of a function $f:K\propto N \rightarrow C$ means that $f(k,x)$ depends only on $x$. Accordingly, if one defines $f_{N}:N\rightarrow C$ by $f_{N}(x)=f(e,x)$, then one obtains a bijection $L^{1}(K\propto N//K)\cong L_{K}^{1}(N)$ given by $f\rightarrow f_{N}$. Here $L_{K}^{1}(N)$ denotes the $K$-invariant functions on $N$, i.e. those $f\in L_{K}^{1}(N)$ such that $f(k.x)=f(x)$ for all $x\in N$ and $k\in K$. One verifies easily that this map respects the conclusion product and we see that $K\subseteq K\propto N$ is a Gelfand pair if and only if, the convolution algebra $L_{K}^{1}(N)$ is commutative. Thus, the definition given here agrees with the more standard one.

Note that if $(K_{1},N)$ is a Gelfand pair and $K_{1}\subseteq K_{2}$, $(K_{2},N)$ is also a Gelfand pair. Also note that we can assume that $K$ acts
faithfully on $N$ since we can always replace $K$ by $K/ker\phi$. In this way, we can regard $K$ as a compact subgroup of $Aut(N)$.

\begin{lemma}\label{twist}
[7]Let $K$,$L$ be compact groups acting on $G$ which are conjugate inside $Aut(G)$. Then $(K,G)$ is a Gelfand pair if,and only if, $(L,G)$ is a Gelfand pair.
\end{lemma}

For the (2a+1)-dimensional Heisenberg group $H_{a}$, the natural action of the group of $a\times a$ unitary matrices on $C^{a}$(which we denote by $k.z$ for $k\in U(a)$ and $Z\in C^{a}$ gives rise to a compact subgroup of $Aut(H_{a})$ via $k.(z,t)=(k.z,t)$. This subgroup, again denoted by $U_{a}$, is a maximal connected, compact subgroup of $Aut(H_{a})$ and thus any connected, compact subgroup of $Aut(H_{a})$ is the conjugate of a subgroup $K$ of $U_{a}$. Since conjugates of $K$ form Gelfand pairs with $H_{a}$ if and only if, $K$ does, and produce the same spherical functions, I will always assume that I am dealing with a compact subgroup of $U(a)$.

\begin{theorem}\label{equal}
[7] Let $N$ be a connected, simply connected nilpotent lie group. If $N$ is an $n$-step group with $n\geq 3$ then there are no Gelfand pairs $(K,N)$.
\end{theorem}

Suppose the (2a+1)-dimensional Heisenberg group $H_{a}$ has lie algebra $h_{a}$ with basis $X_{1},\cdots ,X_{a},Y_{1},\cdots ,Y_{a},Z$ and structure equations given by $[X_{i},Y_{i}]=Z$. We are call the representation theory of $H_{a}$. A generic set of coadjoint orbits in $h_{a}^{*}$ is parametrized by nonzero $\lambda\in R$, where the orbit $O_{\lambda}$ is the hyperplane in $h_{a}^{*}$ of all functions taking the value $\lambda$
at $Z$. The action of $U(a)$ on $h_{a}^{*}$ preserves each $O_{\lambda}$. Hence, if $\pi(\lambda)$ is the element of $\widehat{H_{a}}$ corresponding
to $O_{\lambda}$, the $U(a)$ also preserves the equivalence class of $\pi(\lambda)$. One can realize $\pi(\lambda)$ in the Fock space.

The Fock model, for real $\lambda>0$, is defined on the space $F_{\lambda}$ of holomorphic functions on $C^{a}$ which are square integrable with respect to the measure
$d\widetilde{\omega}_{\lambda}=(\frac{\lambda}{2\pi})^{a}e^{-\frac{\lambda}{2}\left |\omega  \right |^{2}}d\omega d\overline{\omega}$ $[8]$. The space $P(C^{a})$ of holomorphic polynomials is dense in $F_{\lambda}$.

The representation $\pi_{\lambda}$ of $H_{a}$ of $F_{\lambda}$ is given by

$\pi_{\lambda}(z,t)u(\omega)=e^{i\lambda t-\frac{\lambda}{2}<\omega,z >-\frac{\lambda}{4}\left |z  \right |^{2}}u(\omega+z)$

For $\lambda<0$, $F_{\lambda}$ consists of antiholomorphic functions which are square with respect to $d\widetilde{\omega}_{\left |\lambda \right |}$,
and the representation is given by

$\pi_{\lambda}(z,t)u(\overline{\omega})=e^{i\lambda t+\frac{\lambda}{2}<\omega,z >+\frac{\lambda}{4}\left |z  \right |^{2}}u(\overline{\omega+z})$

Since the irreducible unitary representations of $H_{a}$ which are non-trivial on the center $R$, are determined up to equivalence by their centre character. For $k\in U(a)$, the representation $\pi_{\lambda}^{k}(z,t)=\pi_{\lambda}(k.z,t)$ has the same centre character as $\pi_{\lambda}$, and hence is equivalent to $\pi_{\lambda}$. For $\lambda>0$, the operator that intertwines these two representations comes from the standard action of $U(a)$ on $C^{a}$. More precisely,

$[\pi_{\lambda}(k.z,t)u](k.\omega)=[\pi_{\lambda}(z,t)(k^{-1}.u)](\omega)$.
where $k\times u(\omega)=u(k^{-1}\times \omega)$. One has a similar formula for $\lambda<0$, except that the action of $U(a)$ on antiholomorphic functions is given by $k\times u(\overline{\omega})=u(\overline{k\times \omega})$.

If we denote $W_{\lambda}(k)u(\omega)=u(k^{-1}.\omega)$ for $\lambda>0$ and $W_{\lambda}(k)u(\overline{\omega})=u(\overline{k.\omega})$ for $\lambda<0$
respectively. We have $W_{\lambda}(k)\pi_{\lambda}(z,t)W_{\lambda}(k)^{-1}=\pi_{\lambda}(k.z,t)$. That is, $U(a)$ is the stablizer of the equivalence class of $\pi_{\lambda}\in \widehat{H_{a}}$ under the action of $U(a)$ and $W_{\lambda}$. I remark that up to a factor of $det(k)^{\frac{1}{2}}$, $W_{\lambda}$ lifts to the oscillator representation on the double cover $MU(a)$ of $U(a)$ (cf[9]) Now I give another way to introduce the oscillator representation. It is as follows.

\begin{theorem}\label{equal}
(Moore-Wolf). An irreducible representation $\sigma$ of a nilpotent Lie group $N$ is square-integrable(modulo the centre of $N$)if and only if it is the unique irreducible representation of $N$ with its central character.
\end{theorem}

Let $\alpha$ be an automorphism of $H$ which acts trivially on $T$. By the uniqueness of $\rho$ we know from general considerations that there is a unitary operator $\omega(\alpha)$, defined up to a scalar multiples, such that
$\omega(\alpha)\rho(h)\omega(\alpha)^{-1}=\rho(\alpha(h))$.

\begin{theorem}\label{equal}
(Shale-Weil). Let $\widetilde{S}p$ be the 2-fold cover of Sp. Let $\widetilde{g}\rightarrow g$ be the projection map. Then there is a unitary representation $\omega$ of $\widetilde{S}p$ on $L^{2}$ such that
\end{theorem}
$\omega(\widetilde{g})\rho(h)\omega(\widetilde{g})^{-1}=\rho(g(h))$.
I call $\omega$ the oscillator representation.

Given a compact, connected subgroup $K\subseteq U(n)$, we denote its complexification by $K_{C}$. The action of $K$ on $C^{n}$ yields a representation of $K_{C}$ on $C^{n}$, and one can view $K_{C}$ as a subgroup of $Gl(n,C)$.

A finite dimensional representation $\rho:G\rightarrow Gl(V)$ in a complex vector space $V$ is said to be multiplicity free if each irreducible $G$-modules occurs at most once in the associated representation on the polynomial ring $C[V]$.(given by $(x.p)(z)=p(\rho(x^{-1})z))$.

\begin{theorem}\label{equal}
Let $K$ be a compact, connected subgroup of $U(a)$ acting irreducibly on $C^{a}$. The following are equivalent:(i)$(K,H_{a})$ is a Gelfand pair.(ii)The representation of $K_{C}$ on $C^{a}$ is multiplicity free. (iii)The representation of $K_{C}$ on $C_{a}$ is equivalent to one of the representations in a table[7].
\end{theorem}

\begin{theorem}\label{equal}
[7]$(SO(n),F(n))$ and $(O(n),F(n))$ are Gelfand pairs for all $n\geq 2$.
\end{theorem}

\begin{theorem}\label{equal}
[7]If $K$ is a proper, closed(not necessarily connected) subgroup of $SO(n)$ then $(K,F(n))$ is not a Gelfand pair.
\end{theorem}

A result due to Howe and Umeda (cf. [10]) shows that
$\mathbb{C}[v_{R}]^{K}$ is freely
generated as an algebra. So there are polynomials
$\gamma_{1},\cdots ,\gamma_{d}\in \mathbb{C}[v_{R}]^{K}$ so that
$\mathbb{C}[v_{R}]^{K}=\mathbb{C}[\gamma_{1},\cdots ,\gamma_{d}]$.

We call $\gamma_{1},\cdots ,\gamma_{d}$ the fundamental invariants and we can suppose $\gamma_{1}(z)=\gamma(z)=\frac{\left |z \right |^{2}}{2}$.

Invariant different operators. The algebra $\mathbb{D}(H_{a})$ of
left-invariant
differential operators on $H_{a}$ is generated by
$\{Z_{1},\cdots ,Z_{a},\overline{Z_{1}},\cdots ,\overline{Z_{a}},T\}$. We
denote the subalgebra of $K$-invariant differential operators by

$\mathbb{D}_{K}(H_{a}):=\{D\in \mathbb{D}(H_{a})\mid D(f\circ k)=D(f)\circ
k \ for \ k\in K,f\in C^{\infty}(H_{a})\}$

From now on, we always suppose $(K,H_{a})$ is a Gelfand pair, and if this
is true, $\mathbb{D}_{K}(H_{a})$ is an abelian algebra.

One differential operator will play a key role in the Heisenberg group. This is the Heisenberg sub-Laplacian defined by

$U=\frac{1}{2}\sum_{j=1}^{n}(Z_{j}\overline{Z}_{j}+\overline{Z}_{j}Z_{j})$.
$U$ is $U(a)$-invariant and hence belongs to $\mathbb{D}_{K}(H_{a})$ for all Gelfand pairs $(K,H_{a})$. Note that $U$ is essentially self-adjoint on $L^{2}(H_{a})$.

The eigenvalues of the Heisenberg sub-Laplacian $U$ on the type 1 $K$-spherical functions are given by

$U(\phi_{\alpha,\lambda})=-\left |\lambda \right |(2\left |\alpha  \right |+a)\phi_{\alpha,\lambda}$.

Define $\phi_{\alpha,\lambda}$ for $\alpha\in \wedge$ and $\lambda\in
\mathbb{R}^{\times}$ by

$\phi_{\alpha,\lambda}(z,t)=\phi_{\alpha}(\sqrt{\left |\lambda \right |}z,\lambda t)$,

so that $\phi_{\alpha}=\phi_{\alpha,1}$. The $\phi_{\alpha,\lambda}$'s are dinstinct bounded $K$-spherical functions. We refer to these elements of $\Delta(K,H_{a})$ as the spherical function of type 1. One can show that $\phi_{\alpha}$ has the general form

$\phi_{\alpha}(z,t)=e^{it}q_{\alpha}(z)e^{-\frac{\left |z \right |}{4}}$,

where $q_{\alpha}$ is a $K$-invariant polynomial on $V_{R}$ with homogeneous component of highest degree given by
$(-1)^{\left |\alpha \right |}p_{\alpha}/dim(P_{\alpha})$.

In addition to the $K$-spherical functions of type 1, there are $K$-spherical functions which arise from the one-dimensional representations of $H_{a}$. For $\omega \in V$, let

$\eta_{\omega}(z,t)=\int_{K}e^{iRe<\omega,k.z>}dk=\int_{K}e^{iRe<z,k.\omega>}dk$

where $"dk"$ denotes normalized Haar measure on $K$. The $\eta_{\omega}$ are the bounded $K$-spherical functions of type 2. Note that $\eta_{0}$
is the constant function 1 and $\eta_{\omega}=\eta_{\omega^{'}}$ if and only if $K.\omega=K.\omega^{'}$. It is shown in [16] that every bounded $K$-spherical function is of type 1 or type 2. Thus we have:

\begin{theorem}\label{equal}
The bounded $K$-spherical functions on $H_{a}$ are parametrized by the set
$(\mathbb{R}^{\times}\times \wedge)\cup (V/K)$ via

$\bigtriangleup(K,H_{a})=\{\phi_{\alpha,\lambda }\mid \lambda\in
\mathbb{R}^{\times},\alpha\in \wedge\}\cup \{\eta_{K_{\omega}}\mid
\omega\in V\}$

Note that, for $\psi\in \bigtriangleup(K,H_{n})$, one has
$\psi(z,t)=e^{i\lambda t}\psi(z,0)$,

where $\lambda=-i\widehat{T}(\psi)\in \mathbb{R}$.
\end{theorem}

Finally, we consider the two-step free nilpotent lie groups.
First Definition. Let $\mathcal{N}$ be the (unique up to isomorphism)
free two-step nilpotent Lie algebra with $n$ generators. The definition
using the universal property of the free nilpotent Lie algebra can be
found in [11, Chapter V §5]. Roughly speaking, $\mathcal{N}$ is a
(nilpotent)Lie algebra with $n$ generators $X_{1},\cdots X_{n}$, such that
the vectors $X_{1},\cdots X_{n}$ and $X_{i,j}=[X_{i},X_{j}],i<j$ form a
basis; we call this basis the canonical basis of $\mathcal{N}$.

We denote by $\mathcal{V}$ and $\mathcal{Z}$, the vectors spaces generated
by the families of vectors $X_{1},\cdots X_{n}$ and
$X_{i,j}=[X_{i},X_{j}], 1\leqslant i<j \leqslant n$ respectively; these
families become the canonical base of $\mathcal{V}$ and $\mathcal{Z}$.
Thus $\mathcal{N}=\mathcal{V}\bigoplus \mathcal{Z}$, and $\mathcal{Z}$
is the center of $\mathcal{N}$. With the canonical basis, the vector
space $\mathcal{Z}$ can be identified with the vector space of
antisymmetric $n\times n$-matrices $\mathcal{A}_{n}$. Let
$z=dimZ=n(n-1)/2$.

The connected simply connected nilpotent Lie group which corresponds to
$\mathcal{N}$ is called the free two-step nilpotent Lie group and is
denoted $F_{n}$. We denote by exp:$\mathcal{N}\rightarrow F(n)$the
exponential map.

In the following, we use the notations $X+A\in \mathcal{N}$,$exp(X+A)\in
F(n)$ when $X\in \mathcal{V},A\in \mathcal{Z}$. We write $n=2p^{'} or
2p^{'}+1$.

A Realization of $\mathcal{N}$. We now present here a realization of
$\mathcal{N}$, which will be helpful
to define more naturally the action of the orthogonal group and
representations of $F(n)$.

Let $(\mathcal{V},<,>)$ be an Euclidean space with dimension $n$. Let
$O(\mathcal{V})$ be the group of
orthogonal transformations of $\mathcal{V}$, and $SO(\mathcal{V})$ Its
special subgroup. Their common
Lie algebra denoted by $\mathcal{Z}$, is identified with the vector space
of antisymmetric transformations of $\mathcal{V}$. Let
$\mathcal{N}=\mathcal{V}\bigoplus \mathcal{Z}$ be the exterior direct sum
of the vector spaces $\mathcal{V}$ and $\mathcal{Z}$.

Let $[,]:\mathcal{V}\times \mathcal{V}\rightarrow \mathcal{Z}$ be the
bilinear application given by:

$[X,Y].(V)=<X,V>Y-<Y,V>X   \\ where X,Y,V\in \mathcal{V}$

We also denote by $[,]$ the bilinear application extended to
$\mathcal{N}\times \mathcal{N}\rightarrow \mathcal{N}$ by:

$[.,.]_{\mathcal{N}\times \mathcal{Z}}=[.,.]_{\mathcal{Z}\times \mathcal{N}}=0$
This application is a Lie bracket. It endows $\mathcal{V}$ with the
structure of a two-step nilpotent Lie algebra.

As the elements $[X,Y],X,Y\in \mathcal{V}$ generate the vector space
$\mathcal{Z}$, we also
define a scalar product $<，>$ on $\mathcal{Z}$ by:

$<[X,Y],[X^{'},Y^{'}]>=<X,X^{'}><Y,Y^{'}>-<X,Y^{'}><X^{'},Y>$

where $X,Y,X^{'},Y^{'}\in \mathcal{V}$.

It is easy to see $\mathcal{V}$ as a realization of $\mathcal{N}$ when
an orthonormal basis
$X_{1},\cdots X_{n}$ of $(\mathcal{V},<,>)$ is fixed.

We remark that $<[X,Y],[X^{'},Y^{'}]>=<[X,Y]X^{'},Y^{'}>$, and so we have
for an antisymmetric transformation $A\in \mathcal{Z}$, and for $X,Y\in
\mathcal{V}$:

$<A,[X,Y]>=<A.X,Y>$

This equality can also be proved directly using the canonical basis of
$\mathcal{N}$.

Actions of Orthogonal Groups. We denote by $O(\mathcal{V})$ the group of
orthogonal linear maps of $(\mathcal{V},<,>)$, and by $O_{n}$ the group of
orthogonal $n\times n$-matrices.

On $\mathcal{N}$ and $F(n)$. The group $O(\mathcal{V})$ acts on the
one hand by automorphism on $\mathcal{V}$, on the other hand by the
adjoint representation $Ad_{\mathcal{Z}}$ on $\mathcal{Z}$. We obtain an
action of $O(\mathcal{V})$ on $\mathcal{N}=\mathcal{V}\bigoplus
\mathcal{Z}$. Let us prove that this action respects the Lie bracket of
$\mathcal{N}$. It suffices to show for
$X,Y,Z\in \mathcal{V}$ and $k\in O(\mathcal{V})$:
\begin{equation}\label{l-invariant elements}
\begin{split}
&[k.X,k.Y](V)=<k.X,V>k.Y-<k.Y,V>k.X\\
&k.(<X,k^{t}.V>Y-<Y,k^{t}.V>X\\
&=k.[X,Y](k^{-1}.V)=Ad_{\mathcal{Z}}k.[X,Y].
\end{split}
\end{equation}

We then obtain that the group $O(\mathcal{V})$ and also its special
subgroup $SO(\mathcal{V})$.
acts by automorphism on the Lie algebra $\mathcal{N}$, and finally on the
Lie group $F(n)$.

Suppose an orthonormal basis $X_{1},\cdots X_{n}$ of $(\mathcal{V},<,>)$
is fixed; then the vectors
$X_{i,j}=[X_{i},X_{j}],1\leq i<j\leq n$, form an orthonormal basis of
$\mathcal{V}$ and we can identify:

the vector space $\mathcal{Z}$ and $\mathcal{A}_{n}$.

the group $O(\mathcal{V})$ with $O_{n}$.

the adjoint representation $Ad_{\mathcal{Z}}$ with the conjugate action of
$O_{n}$ and $\mathcal{A}_{n}: k.A=kAk^{-1}$, where $k\in O_{n}, A\in\mathcal{A}_{n}$.

Thus the group $O_{n}\sim O(\mathcal{V})$ acts on $\mathcal{V}\sim
\mathbb{R}^{n}$ and $\mathcal{Z}\sim \mathcal{A}_{n}$, and consequently on
$\mathcal{N}$. Those actions can be directly defined; and the equality
$[k.X,k.Y]=k.[X,Y],k\in O_{n},X,Y\in \mathcal{V}$, can then be computed.

On $\mathcal{A}_{n}$. Now we describe the orbits of the conjugate actions
of $O_{n}$ and $SO_{n}$ on $\mathcal{A}_{n}$. An arbitrary antisymmetric
matrix $A\in \mathcal{A}_{n}$ is $O_{n}$-conjugated to an antisymmetric
matrix $D_{2}(\wedge)$ where $\wedge=(\delta_{1},\cdots
,\delta_{p^{'}})\in \mathbb{R^{'}}$ and:

$D_{2}(\wedge)=\begin{bmatrix}
\ \delta_{1}J &  0   &  0   & 0      \\
0 & \ddots  & 0  &  0   \\
0      &  0   &\delta_{p^{'}}J  &   0       \\
0      &  0    &   0    &   (0)
\end{bmatrix}$

where $J:=\begin{bmatrix}
0      & 1      \\
-1      &  0
\end{bmatrix}$

((0) means that a zero appears only in the case $n=2p^{'}+1$) Furthermore,
we can assume that $\wedge$ is in $\overline{\mathcal{L}}$, where we
denote by $\mathcal{L}$ the set of $\wedge=(\delta_{1},\cdots
,\delta_{p^{'}})\in \mathbb{R^{'}}$ such that $\delta_{1}\geq \cdots
\delta_{p^{'}}\geq 0$.

Parameters. To each $\wedge\in \overline{\mathcal{L}}$, we
associate:$p_{0}$ the number of $\delta_{i}\neq0$, $p_{1}$ the number of
distinct $\delta_{i}\neq0$, and $\mu_{1},\cdots \mu_{p^{1}}$ such that:

$\{\mu_{1}> \mu_{2}> \cdots >\mu_{p_{1}}>0\}=\{\delta_{1}\geq
\delta_{2}\geq \cdots \geq \delta_{p_{0}}>0\}$

We denote by $m_{j}$ the number of $\delta_{i}$ such that
$\delta_{i}=\mu_{j}$, and we put

$m_{0}:=m_{0}^{'}:=0$ and for $j=1,\cdots p_{1} \
m_{j}^{'}:=m_{1}+\cdots+m_{j}$.

For $j=1,\cdots p_{1}$, let $pr_{j}$ be the orthogonal projection of
$\mathcal{V}$ onto the space generated by the vectors $X_{2i-1},X_{2i}$,
for $i=m_{j-1}^{'}+1,\cdots m_{j}^{'}$.

Let $\mathcal{M}$ be the set of $(r,\wedge)$ where $\wedge \in
\mathcal{L}$, and $r\geq 0$, such that $r=0$ if $2p_{0}=n$.

Expression of the bounded spherical functions. The bounded spherical
functions of $(K,F(n))$ for $K=O_{n}$, are parameterized by

$(r,\wedge)\in \mathcal{M}$ (with the previous notations
$p_{0},p_{1},\mu_{i},pr_{j}$ associated to $\wedge$),

$l\in \mathbb{N}^{p_{1}}$ if $\wedge\neq 0$, otherwise $\varnothing$.

Let $(r,\wedge)$,$l$ be such parameters. Then we have the following two
types of bounded $O(n)$-spherical functions:

For $n=exp(X+A)\in N$.

Type 1:$\phi^{r,\alpha,\lambda}(n)=\int_{K}e^{ir<X_{p}^{*},k.X
>}\omega_{\alpha,\lambda}(\Psi_{2}^{-1}(\overline{q_{1}}(k.n)))dk$.

Type 2:$\phi^{\upsilon}(n)=\int_{K}e^{ir<X_{p}^{*},k.X >}dk$.
Here $X_{p}^{*}$ is the unit $K_{\rho}$-fixed invariant vector.
For a Gelfand pair $(H_{p_{0}},K(m;p_{1};p_{0}))$, we have
$\omega_{\alpha,\lambda}$ is the "type 1" bounded $K(m;p_{1};p_{0})$-spherical
functions for the Heisenberg group $H_{p_{0}}$. We will introduce it next.
$\Psi_{2}$ is an isomorphism between $H_{p_{0}}$ with a group with respect
to $F(n)$, which will be introduced later.

We call $H_{p_{0}}$ the Heisenberg group with respect to $F(n)$.

We use the following law of the Heisenberg group $H_{p_{0}}$:

$\forall h=(z_{1},\ldots ,z_{p_{0}},t)$ ,
$h^{'}=(z_{1}^{'},\ldots ,z_{p_{0}}^{'},t^{'})\in \mathbb{H}^{p_{0}}=\mathbb{C}^{p_{0}}\times \mathbb{R}$

$h.h^{'}=(z_{1}+z_{1}^{'},\ldots ,z_{p_{0}}+z_{p_{0}}^{'},t+t^{'}+\frac{1}{2}\sum_{i=1}^{p_{0}}\mathfrak{F}z_{i}\overline{z}_{i}^{'})$

The unitary $p_{0}\times p_{0}$ matrix group $U_{p_{0}}$ acts by automorphisms on $\mathbb{H}^{p_{0}}$. Let us describe some subgroups of $U_{p_{0}}$. Let $p_{0},p_{1}\in \mathbb{N}$, and $m=(m_{1},\ldots ,m_{p_{1}})\in \mathbb{N}^{p_{1}}$ be fixed such that $\sum_{j=1}^{p_{1}}m_{j}=p_{0}$. Let $K(m;p_{1};p_{0})$ be the subgroup of $U_{p_{0}}$ given by:

$K(m;p_{1};p_{0})=U_{m_{1}}\times \ldots \times U_{m_{p_{1}}}$.

The expression of spherical functions of $(H_{p_{0}},K(m;p_{1};p_{0}))$ can be found in the same way as in the case $m=(p_{0})$, $p_{1}=1$
i.e. $K=U_{p_{0}}$ (cf.[12]).

Stability group $K_{\rho}=\{k\in K: k.\rho=\rho\}=\{k\in K\subset G: k.f\in F(n).f\}$. The aim of this paragraph is to describe the stability group $K_{\rho}$
of $\rho\in T_{rX_{p}}^{*}+D_{2}(\wedge)$.

Before this, let us recall that the orthogonal $2n\times 2n$ matrices which commutes with $D_{2}(1,\ldots ,1)$ have determinant one and form the group$Sp_{n}\bigcap O_{2n}$. This group is isomorphism to $U_{n}$; the isomorphism is denoted $\psi_{1}^{(n)}$, and satisfies:

$\forall k,X$: $\psi_{c}^{(n)}(k.X)=\psi_{1}^{(n)}(K)\psi_{c}^{(n)}(X)$,

where $\psi_{c}^{(n)}$ is the complexification :

$\psi_{c}^{(n)}(x_{1},y_{1};\ldots ;x_{n},y_{n})=(x_{1}+iy_{1},\ldots ,x_{n}+iy_{n})$.

Now, we can describe $K_{\rho}$:
\begin{prop}
Let $(r,\Lambda)\in \mathcal{M}$. Let $p_{0}$ be the number of $\lambda_{i}\neq 0$, where $\wedge=(\lambda_{1},\ldots ,\lambda_{p^{'}})$, and $p_{1}$ the number of distinct $\lambda_{i}\neq 0$. We set $\widetilde{\wedge}=(\lambda_{1},\ldots ,\lambda_{p_{0}})\in \mathbb{R}^{p_{0}}$.
Let $\rho\in T_{f}$ where $f=rX_{p}^{*}+D_{2}(\wedge)$.

If $\wedge=0$, then $K_{\rho}$ is the subgroup of $K$ such that $k.rX_{p}^{*}=rX_{p}^{*}$ for all $k\in K_{\rho}$.

If $\wedge\neq 0$, then $K_{\rho}$ is the direct product $K_{1}\times K_{2}$, where:

$K_{1}=\{k_{1}=\begin{bmatrix}
\widetilde{k_{1}}     &      0      \\
0 & Id
\end{bmatrix} \mid \widetilde{k_{1}}\in SO(2p_{0}) \  D_{2}(\widetilde{\wedge})\widetilde{k}_{1}=\widetilde{k}_{1}D_{2}(\widetilde{\wedge})\}$

$K_{2}=\{k_{2}=\begin{bmatrix}
Id    &      0      \\
0 & \widetilde{k_{1}}
\end{bmatrix} \mid \widetilde{k_{2}}.rX_{p}^{*}=rX_{p}^{*}\}$.

Furthermore, $K_{1}$ is isomorphism to the group $K(m;p_{0};p_{1})$.
\begin{proof}
We keep the notations of this proposition, and we set $A^{*}=D_{2}(\wedge)$ and $X^{*}=rX_{p}^{*}$. It is easy to prove:

$K_{\rho}=\{k\in K: kA^{*}=A^{*}k \ and \ kX^{*}=X^{*}k\}$.

If $\wedge=0$, since $K_{\rho}$ is the stability group in $K$ of $X^{*}\in \mathcal{V}^{*}\sim \mathbb{R}^{n}$. So the
first part of Proposition 2.12 is proved.

Let us consider the second part. $\wedge\neq 0$ so we have

$A^{*}=\begin{bmatrix}
D_{2}(\widetilde{\wedge})    &     0      \\
0      &  0
\end{bmatrix}$ \ with \ $D_{2}(\widetilde{\wedge})=\begin{bmatrix}
\mu_{1}J_{m_{1}}     & 0 & 0      \\
0 & \ddots & 0 \\
0      & 0  & \mu_{p_{1}}J_{m_{p_{1}}}
\end{bmatrix}$

Let $k\in K_{\rho}$. From above computation, the matrices $k$ and $A^{*}$ commute and we have:

$k=\begin{bmatrix}
\widetilde{k}_{1}   &    0      \\
0      & \widetilde{k}_{2}
\end{bmatrix}$ \ with \ $\widetilde{k}_{1}\in O(2p_{0})$ \ and \ $\widetilde{k}_{2}\in O(n-2p_{0})$

furthermore, $\widetilde{k}_{2}.X^{*}=X^{*}$, and the matrices $\widetilde{k}_{1}$ and $D_{2}(\widetilde{\wedge}^{*})$ commute. So
$\widetilde{k}_{1}$ is the diagonal block matrix, with block $[\widetilde{k}_{1}]_{j}\in O(m_{j})$ for $i=1, \ldots ,p_{1}$. Each block
$[\widetilde{k}_{1}]_{j}\in O(m_{j})$ commutes with $J_{m_{j}}$. So on one hand, we have det$[\widetilde{k}_{1}]_{j}=1$, det$\widetilde{k}_{1}=1$,
and one the other hand, $[\widetilde{k}_{1}]_{j}\in O(m_{j})$ corresponds to a unitary matrix $\psi_{1}^{(m_{j})}([\widetilde{k}_{1}]_{j})$. Now
we set for $k_{1}\in K_{1}$:

$\Psi_{1}(k_{1})=(\psi_{1}^{(m_{j})}([\widetilde{k}_{1}]_{1}),\ldots ,\psi_{1}^{(m_{j})}([\widetilde{k}_{1}]_{p_{1}}))$

$\Psi_{1}:K_{1}\longrightarrow K(m;p_{0};p_{1})$ is a group isomorphism.
\end{proof}
\end{prop}

Quotient group $\overline{F(n)}=F(n)/ker\rho$. In this paragraph, we describe
the quotient groups $F(n)/ker\rho$ and $G/ker\rho$, for some $\rho\in
\widehat{F(n)}$. This will permit in the next paragraph to reduce the
construction of the bounded spherical functions on $F(n)$ to known
questions on Euclidean and Heisenberg groups. For a representation
$\rho\in \widehat{F(n)}$, we will denote by:

$ker\rho$ the kernel of $\rho$.

$\overline{F(n)}=F(n)/ker\rho$ its quotient group and $\overline{N}$ its lie algebra.

$(\mathcal{H},\overline{\rho})$ the induced representation on $\overline{F(n)}$.

$\overline{n}\in \overline{F(n)}$ and $\overline{Y}\in
\overline{\mathcal{N}}$ the image of $n\in F(n)$ and $Y\in \mathcal{N}$
respectively by the canonical projections $F(n)\rightarrow \overline{F(n)}$ and
$\mathcal{N}\rightarrow \overline{\mathcal{N}}$

Now, with the help of the canonical basis, we choose

$E_{1}=\mathbb{R}X_{1}\bigoplus \cdots \bigoplus  \mathbb{R}X_{2p_{0}-1}$

as the maxiamal totally isotropic space for $\omega_{D_{2}(\wedge),r}$.
The quotient lie algebra $\overline{\mathcal{N}}$ has the natural
basis:           .
You can refer to [12].

Here, we have denoted
$\left |\wedge  \right |=(\sum_{j=1}^{p^{'}}\lambda_{j}^{2})^{\frac{1}{2}}=\left |D_{2}(\wedge)
\right |$ (for the Euclidean norm on $\mathcal{Z}$.

Let $\overline{\mathcal{N}_{1}}$ be the Lie sub-algebra of
$\overline{\mathcal{N}}$, with basis $\overline{X_{1}},\cdots
,\overline{X_{2p_{0}}},\overline{B}$, and $\overline{N_{1}}$ be its
corresponding connected simply connected nilpotent lie group. We define
the mapping : $\Psi_{2}:\mathbb{H}^{p_{0}}\rightarrow
\overline{\mathcal{N}_{1}}$ for $h=(x_{1}+iy_{1},\cdots
,x_{p_{0}}+iy_{p_{0}},t)\in H_{p_{0}}$ by:

$\Psi_{2}(h)=exp(\sum_{j=1}^{p_{0}}\sqrt{\frac{\left |\wedge  \right
|}{\lambda_{j}}}(x_{j}\overline{X_{2j-1}}+y_{j}\overline{X_{2j}})+t\overline{B})$

We compute that each lie bracket of two vectors of this basis equals
zeros, except:

$[\overline{X_{2i-1}},\overline{X_{2i}}]=\frac{\lambda_{i}}{\left |\wedge
\right |}\overline{B},\ i=1, \cdots ,p_{0}$.

From this, it is easy to see:
\begin{theorem}\label{equal}
$\Psi_{2}$ is a group isomorphism between $\overline{N_{1}}$ and
$H_{p_{0}}$
\end{theorem}

Finally, we note that

$\overline{q_{1}}:F(n)\rightarrow \overline{N_{1}}$ is the canonical projection.

Remark that the (Kohn) sub-Laplacian is $L:=-\sum_{i=1}^{p}X_{i}^{2}$ (cf. section 6 in [12]). Then we can deduce that

$L.\phi^{r,\alpha,\lambda}=(\sum_{j=1}^{p_{1}}\delta_{j}(2\alpha_{j}+m_{j})+r^{2})\phi^{r,\alpha,\lambda}$.

\section{some knowledge for the generalized binomial coefficients}

Suppose $(K,H_{a})$ is a Gelfand pair.

Decomposition of $C[V]$.

We decompose $C[V]$ into $K$-irreducible subspaces $P_{\alpha}$,

$C[V]=\sum_{\alpha\in \wedge}P_{\alpha}$ where $\wedge$ is some countably infinite index set. Since the representation of $K$ on $C[V]$ preserves the space
$P_{m}(V)$ of homogeneous polynomials of degree m, each $P_{\alpha}$ is a subspace of some $P_{m}(V)$. We write $\left |\alpha \right |$ for the degree of homogeneity of the polynomials in $P_{\alpha}$, so that $P_{\alpha}\subset P_{\left |\alpha \right |}(V)$. We will write $d_{\alpha}$ for the dimension of
$P_{\alpha}$ and denote by $0\in \wedge$ the index for the scalar polynomials $P_{0}=P_{0}(V)=C$.

For $f:H_{a}\rightarrow C$ we define $f^{\circ}: V\rightarrow C$ by $f^{\circ}(z):=f(z,0)$. We denote $\phi_{\alpha}^{\circ}(z):=\phi_{\alpha,1}(z,0)$.

\begin{theorem}\label{equal}
$\{\phi_{\alpha}^{\circ}\mid \alpha\in \wedge\}$ is a complete orthogonal system in $L_{K}^{2}(V)$ with
$\left \|\phi_{\alpha}^{\circ}  \right \|_{2}^{2}=\frac{(2\pi)^{a}}{d_{\alpha}}$.
\end{theorem}

For $\alpha,\beta\in \wedge$, we have a well defined number
$\begin{bmatrix}
\alpha\\
\beta
\end{bmatrix}$.

We call the value generalized binomial coefficients for the action of $K$ on $V$ (cf.[24]).  We have two important results, they are as follows:

$\sum_{\left |\beta \right |=\left |\alpha \right |-1}\begin{bmatrix}
\alpha\\
\beta
\end{bmatrix}=\left |\alpha \right |$

$\sum_{\left |\beta \right |=\left |\alpha \right |+1}\frac{d_{\beta}}{d_{\alpha}}\begin{bmatrix}
\beta\\
\alpha
\end{bmatrix}=\left |\alpha \right |+a$.

\begin{definition}
Given a function $g$ on $\wedge$, $D^{+}g$ and $D^{-}g$ are the functions on $\wedge$ defined by

$D^{+}g(\alpha)=\sum_{\left |\beta \right |=\left |\alpha \right |+1}\frac{d_{\beta}}{d_{\alpha}}\begin{bmatrix}
\beta\\
\alpha
\end{bmatrix}g(\beta)-(\left |\alpha \right |+a)g(\alpha)$

=$\sum_{\left |\beta \right |=\left |\alpha \right |+1}\frac{d_{\beta}}{d_{\alpha}}\begin{bmatrix}
\beta\\
\alpha
\end{bmatrix}(g(\beta)-g(\alpha))$

$D^{-}g(\alpha)=\left |\alpha \right |g(\alpha)-\sum_{\left |\beta \right |=\left |\alpha \right |-1}\begin{bmatrix}
\alpha\\
\beta
\end{bmatrix}g(\beta)$

=$\sum_{\left |\beta \right |=\left |\alpha \right |-1}\begin{bmatrix}
\alpha\\
\beta
\end{bmatrix}(g(\alpha)-g(\beta))$

for $\left |\alpha \right |>0$, and $D^{-}g(0)=0$.
\end{definition}

Also we can compute $\gamma \phi_{\alpha}^{\circ}=-(D^{+}-D^{-})\phi_{\alpha}^{\circ}$,

$\gamma \phi_{\alpha,\lambda}=-\frac{1}{\left |\lambda \right |}(D^{+}-D^{-})\phi_{\alpha,\lambda}$.

Thus we can compute
\begin{equation}\label{l-invariant elements}
\begin{split}
\partial_{\lambda}\phi_{\alpha,\lambda}=
\begin{cases}
(\frac{1}{\lambda})D^{-}\phi_{\alpha,\lambda}-(\frac{\gamma}{2})\phi_{\alpha,\lambda}+it\phi_{\alpha,\lambda} \\
(\frac{1}{\lambda})D^{+}\phi_{\alpha,\lambda}+(\frac{\gamma}{2})\phi_{\alpha,\lambda}+it\phi_{\alpha,\lambda}
\end{cases}
\end{split}
\end{equation}
for $\lambda>0$.

and similarly
\begin{equation}\label{l-invariant elements}
\begin{split}
\partial_{\lambda}\phi_{\alpha,\lambda}=
\begin{cases}
(\frac{1}{\lambda})D^{-}\phi_{\alpha,\lambda}+(\frac{\gamma}{2})\phi_{\alpha,\lambda}+it\phi_{\alpha,\lambda} \\
(\frac{1}{\lambda})D^{+}\phi_{\alpha,\lambda}-(\frac{\gamma}{2})\phi_{\alpha,\lambda}+it\phi_{\alpha,\lambda}
\end{cases}
\end{split}
\end{equation}
for $\lambda<0$.

Equivalently
\begin{equation}\label{l-invariant elements}
\begin{split}
(\frac{\gamma}{2}+it)\phi_{\alpha,\lambda}=\begin{cases}
(\partial_{\lambda}-\frac{1}{\lambda}D^{+})\phi_{\alpha,\lambda},  & \mbox{for }\lambda \mbox{is greater than 0} \\
(\partial_{\lambda}-\frac{1}{\lambda}D^{-})\phi_{\alpha,\lambda},  & \mbox{for }\lambda \mbox{is smaller than 0}
\end{cases}
\end{split}
\end{equation}
and
\begin{equation}\label{l-invariant elements}
\begin{split}
(\frac{\gamma}{2}-it)\phi_{\alpha,\lambda}=\begin{cases}
-(\partial_{\lambda}-\frac{1}{\lambda}D^{-})\phi_{\alpha,\lambda},  & \mbox{for }\lambda \mbox{is greater than 0} \\
-(\partial_{\lambda}-\frac{1}{\lambda}D^{+})\phi_{\alpha,\lambda},  & \mbox{for }\lambda \mbox{is smaller than 0}
\end{cases}
\end{split}
\end{equation}
\begin{definition}
We say that a function $F:\wedge \rightarrow C$ is rapidly decreasing if for each $N\in Z^{+}$, there is a constant $C_{N}$ for which

$\left |F(\alpha) \right |\leq \frac{C_{N}}{(2\left |\alpha \right |+a)^{N}}$.
\end{definition}

\begin{theorem}\label{equal}
[13] If $f\in \varphi_{K}(V)$ then $\widehat{f}$ is rapidly decreasing on $\wedge$. Conversely, if $F$ is rapidly decreasing on $\wedge$ then $F=\widehat{f}$
for some $f\in \varphi_{K}(V)$. Moreover, the map

$\land:\varphi_{K}(V)\rightarrow \{F\mid F $ is rapidly decreasing on $\wedge\}$
is a bijection.
\end{theorem}

\begin{lemma}
[13] Let $F$ be a rapidly decreasing function on $\wedge$ and $G$ be a bounded function on $\wedge$. Then

$\sum_{\alpha\in \wedge}d_{\alpha}F(\alpha)D^{+}G(\alpha)=-\sum_{\alpha\in \wedge}d_{\alpha}(D^{-}+a)F(\alpha)G(\alpha)$,

$\sum_{\alpha\in \wedge}d_{\alpha}F(\alpha)D^{-}G(\alpha)=-\sum_{\alpha\in \wedge}d_{\alpha}(D^{+}+a)F(\alpha)G(\alpha)$,
\end{lemma}

\section{The proof of the third main theorem}

In this section, we identity $H_{a}$ wtih $H_{p_{0}}$ and $\phi_{\alpha,\lambda}$ with $\omega_{\alpha,\lambda}$.

\begin{definition}
Let $G$ be a function on $\Delta (O(n),F(n))$. We say that $G$ is rapidly decreasing on $\Delta (O(n),F(n))$ if

1 $G$ is continuous on $\Delta (O(n),F(n))$.

2 The function $G_{0}$ on $R$ defined by $G_{0}(r)=G(\phi^{r})$ belongs to $\phi(R)$.

3 The map $\lambda\rightarrow G(\phi^{r,\alpha,\lambda})$ is smooth on $R^{\times}=(-\infty,0)\cup (0,\infty)$ for each fixed $\alpha\in \wedge$ and $r\in R$.

4 for each $m,N\geq 0$, there exists a constant $C_{m,N}$ for which

$\left |\partial_{\lambda}^{m}G(\phi^{r,\alpha,\lambda}) \right |\leq \frac{C_{m,N}}{\left |\lambda \right |^{m}(\sum_{j=1}^{p_{1}}\delta_{j}(2\alpha_{j}+m_{j})+r^{2})^{N}}$ for all $(r,\alpha,\lambda)\in R\times \wedge \times R^{\times}$.
\end{definition}

We say that a continuous function on $\Delta_{1}(O(n),F(n))$ is rapidly decreasing if it extends to a rapidly decreasing function on
$\Delta (O(n),F(n))=\Delta_{1}(O(n),F(n))\cup  \Delta_{2}(O(n),F(n))$. Since $\Delta_{1} (O(n),F(n)$ is dense in $\Delta (O(n),F(n)$, such an extension is necessarily unique.

Note that if $G$ is rapidly decreasing on $\Delta (O(n),F(n))$, then $\alpha \rightarrow G(\phi^{r,\alpha,\lambda})$ is rapidly decreasing on $\wedge$, in the sense of Definition 4.1 for each $\lambda \neq 0$. We see that $F$ is bounded by letting $m=N=0$ and one can show, moreover, that $G$ vanishes at infinity by letting $m=0$ and $N=1$. We remark that the functions $\partial_{\lambda}^{m}G(\phi^{r,\alpha,\lambda})$ defined on $\Delta_{1}(O(n),F(n))$ need not extend continuously across $\Delta_{2}(O(n),F(n))$.

\begin{definition}
Let $G$ be a function on $\Delta_{1}(O(n),F(n))$ which is smooth in $\lambda$. $M^{+}G$ and $M^{-}G$ are the functions on $\Delta_{1}(O(n),F(n))$ defined by

$M^{+}G(\phi^{r,\alpha,\lambda})=\begin{cases}
(\partial_{\lambda}-\frac{1}{\lambda}D^{+})G(\phi^{r,\alpha,\lambda}),  & \mbox{for }\lambda \mbox{is greater than 0} \\
(\partial_{\lambda}-\frac{1}{\lambda}D^{-})G(\phi^{r,\alpha,\lambda}),  & \mbox{for }\lambda \mbox{is smaller than 0}
\end{cases}$

and

$M^{-}G(\phi^{r,\alpha,\lambda})=\begin{cases}
(\partial_{\lambda}-\frac{1}{\lambda}D^{-})G(\phi^{r,\alpha,\lambda}),  & \mbox{for }\lambda \mbox{is greater than 0} \\
(\partial_{\lambda}-\frac{1}{\lambda}D^{+})G(\phi^{r,\alpha,\lambda}),  & \mbox{for }\lambda \mbox{is smaller than 0}
\end{cases}$
\end{definition}

We reminded the reader that the difference operators $D^{\pm}$ are defined by

$D^{+}G(\phi^{r,\alpha,\lambda})=\sum_{\left |\beta \right |=\left |\alpha \right |+1}\frac{d_{\beta}}{d_{\alpha}}\begin{bmatrix}
\beta\\
\alpha
\end{bmatrix}G(\phi^{r,\beta,\lambda})-(\left |\alpha \right |+a)G(\phi^{r,\alpha,\lambda})$

$D^{-}G(\phi^{r,\alpha,\lambda})=\left |\alpha \right |G(\phi^{r,\alpha,\lambda})-\sum_{\left |\beta \right |=\left |\alpha \right |-1}\begin{bmatrix}
\alpha\\
\beta
\end{bmatrix}G(\phi^{r,\beta,\lambda})$.

\begin{definition}
$\widehat{\varphi}(O(n),F(n))$ is the set of all functions $G:\Delta(O(n),F(n))\rightarrow C$ for which $(M^{+})^{l}(M^{-})^{m}G$ is rapidly decreasing for all $l, m\geq 0$.

\end{definition}

If $G$ is rapidly decreasing on $\Delta(O(n),F(n))$ then $\lambda\rightarrow G(\phi^{r,\alpha,\lambda})$ is smooth on $R^{\times}$ and we have well defined functions $(M^{+})^{l}(M^{-})^{m}G$ on $\Delta_{1}(O(n),F(n))$. $G$ belongs to $\widehat{\varphi}(O(n),F(n))$ if and only if these functions extend continuously to rapidly decreasing functions on $\Delta(O(n),F(n))$.

Note that $f(x)=\int_{\Delta(O(n),F(n))}\widehat{f}(\psi)\psi(x)d\mu(\psi)$
Here $x=exp(X+A)$.

\begin{theorem}\label{equal}
(The Main Theorem) If $f\in \varphi_{O(n)}(F(n))$ then $\widehat{f}\in \widehat{\varphi}(O(n),F(n))$. Conversely, if $G\in \widehat{\varphi}(O(n),F(n))$, then
$G=\widehat{f}$ for some $f\in \varphi_{O(n)}(F(n))$. Moreover, the map $\land: \varphi_{O(n)}(F(n))\rightarrow \widehat{\varphi}(O(n),F(n))$ is a bijection.
\end{theorem}

If $f\in \varphi_{O(n)}(F(n))$ and $\widehat{f}=0$ then the inverse formula for the spherical transform yields that $f=0$. Thus the spherical transform is injective on $\varphi_{O(n)}(F(n))$. To prove above theorem, it remains to show that $\varphi_{O(n)}(F(n))^{\land}\subset \widehat{\varphi}(O(n),F(n))$, and that
$\widehat{\varphi}(O(n),F(n))\subset \varphi_{O(n)}(F(n))^{\land}$.

Proof of $\varphi_{O(n)}(F(n))^{\land}\subset \widehat{\varphi}(O(n),F(n))$. Suppose that $f\in \varphi_{O(n)}(F(n))$ and let $G:=\widehat{f}$. We begin by showing
that $G$ is rapidly decreasing. $G$ is continuous on $\Delta (O(n),F(n))$, as is the spherical transform of any integrable $O(n)$-invariant function. Moreover,
$G_{0}(r)=G(\phi^{r})=\int_{F(n)}f(x)\overline{\phi^{r}(x)}dx=\int_{F(n)}f(x)\int_{K}e^{-ir<X_{p}^{*},k.X>}dx=\int_{F(n)}f(x)e^{-ir<X_{p}^{*},X>}dx$. Since
$f$ is a Schwartz function, so is $G_{0}(r)$. Thus $G$ satisfies the first two conditions in Definition 3.2.1.

Next, I will show that $G$ satisfies the estimates in Definition 4.1 for $m=0$. Recall that the (Kohn) Sub-Laplacian $L$ is a self-adjoint operator on $L^{2}(F(n))$ with

$L(\phi^{r,\alpha,\lambda})=(\sum_{j=1}^{p_{1}}\delta_{j}(2\alpha_{j}+m_{j})+r^{2})\phi^{r,\alpha,\lambda}$. (cf. [12])

WLOG, I denote $G(\phi^{r,\alpha,\lambda})=G(r,\alpha,\lambda)$.

Thus we have

$(\sum_{j=1}^{p_{1}}\delta_{j}(2\alpha_{j}+m_{j})+r^{2})^{N}\left |G(r,\alpha,\lambda) \right |$

=$\left |(\sum_{j=1}^{p_{1}}\delta_{j}(2\alpha_{j}+m_{j})+r^{2})^{N}<f,\phi^{r,\alpha,\lambda}>_{2} \right |$

=$\left |<f,L^{N}\phi^{r,\alpha,\lambda}>_{2} \right |=\left |<L^{N}f,\phi^{r,\alpha,\lambda}>_{2} \right |$

$\leq \left \|L^{N}f \right \|_{1}$.

Since $\left |\phi^{r,\alpha,\lambda}(x) \right |\leq 1$ for all $x\in F(n)$. Letting $C_{0,N}:=\left \|L^{N}f \right \|_{1}$, we see that the equality in Definition 4.1 hold for $m=0$.

Since $\phi_{\alpha,\lambda}(z,t)$ is smooth in $R^{\times}$ for fixed $(z,t)$, $\phi^{r,\alpha,\lambda}(x)$ is smooth in $R^{\times}$ for fixed
$x\in F(n)$. And $f\overline{\phi^{r,\alpha,\lambda}}$ is a Schwartz function, $G(r,\alpha,\lambda)=\widehat{f}(r,\alpha,\lambda)$ is smooth in
$\lambda\in R^{\times}$ with

\begin{equation}\label{l-invariant elements}
\begin{split}
& \partial_{\lambda}G(r,\alpha,\lambda)=\int_{F(n)}f(x)\partial_{\lambda}\overline{\phi^{r,\alpha,\lambda}}(x)dx \\
&=\int_{F(n)}f(x)e^{-ir<X_{p}^{*},X>}\partial_{\lambda}\omega_{\alpha,\lambda}(\Psi_{2}^{-1}(\overline{q_{1}(x^{-1})})dx
\end{split}
\end{equation}
Equation 3.3 and 3.4 provide formula for $\partial_{\lambda}\overline{\phi_{\alpha,\lambda}}=\partial_{\lambda}\phi_{\alpha,-\lambda}$ but I
require a different approach here. Note that $\omega_{\alpha,\lambda}(z,t)$ is a special case for $\phi_{\alpha,\lambda}(z,t)$. I write $\omega_{\alpha,\lambda}(z,t)=\omega_{\alpha}^{0}(z)e^{it}$ as $\omega_{\alpha}^{0}(z,\overline{z})e^{it}$
so that $\overline{\omega_{\alpha,\lambda}(z,t)}=\omega_{\alpha}^{0}(\left |\lambda \right |z,\overline{z})e^{-i\lambda t}$.

We see that

$\partial_{\lambda}\overline{\omega_{\alpha,\lambda}(z,t)}$
=$\frac{1}{\lambda}[(\sum_{j=1}^{n}Z_{j}\frac{\partial}{\partial Z_{j}})\omega_{\alpha}^{0}](\left |\lambda \right |z,\overline{z})e^{i\left |\lambda  \right | t}-it\overline{\omega_{\alpha,\lambda}(z,t)}$.

Substituting this expansion onto equation 4.5 and integrating by parts gives

$\partial_{\lambda}G(r,\alpha,\lambda)=\int {F(n)}f(x)e^{-ir<X_{p}^{*},X>}\frac{1}{\lambda}[(\sum_{j=1}^{a}Z_{j}\frac{\partial}{\partial Z_{j}})\omega_{\alpha}^{0}](\left |\lambda \right |z,\overline{z})e^{i\left |\lambda  \right | t}dx$

-$\int_{F(n)}f(x)e^{-ir<X_{p}^{*},X>}i(Pr_{T}(\Psi_{2}^{-1}(\overline{q_{1}(x^{-1})}))\omega_{\alpha,\lambda}(\Psi_{2}^{-1}(\overline{q_{1}(x^{-1})})dx$.

=$\frac{1}{\lambda}(Df)^{\land}(r,\alpha,\lambda)-i(Pr_{T}(\Psi_{2}^{-1}(\overline{q_{1}(x^{-1})})f)^{\land}((r,\alpha,\lambda)$

where $Df=-\sum_{j=1}^{a}\widetilde{\frac{\partial}{\partial Z_{j}}Z_{j}}f$ such that after integration by parts, above equality holds.

Since $Pr_{T}(\Psi_{2}^{-1}(\overline{q_{1}(x^{-1})})f$ and $Df$ are both Schwartz functions on $F(n)$, they satisfy estimates above. Thus, given $N\geq 0$, one can find constants $A$ and $B$ with

$\left |\partial_{\lambda}G(r,\alpha,\lambda) \right |\leq \frac{A}{\left |\lambda \right |(\sum_{j=1}^{p_{1}}\delta_{j}(2\alpha_{j}+m_{j})+r^{2})^{N}}+\frac{B}{(\sum_{j=1}^{p_{1}}\delta_{j}(2\alpha_{j}+m_{j})+r^{2})^{N}}$

$\leq \frac{C_{1,N}}{\left |\lambda \right |(\sum_{j=1}^{p_{1}}\delta_{j}(2\alpha_{j}+m_{j})+r^{2})^{N}}$

where $C_{1,N}=A+\left |\lambda \right |\times B$. By induction on $m$, we see that $\left |\partial_{\lambda}^{m}G(r,\alpha,\lambda) \right |$
satisfies an estimate as in Difinition 4.1. This completes the proof that $F$ is rapidly decreasing. Equations 3.5 and 3.6 shows that

$M^{+}G=((\widetilde{\frac{\gamma}{2}+it}f)^{\land}\mid \Delta_{1} (O(n),F(n))$

and $M^{-}G=-((\widetilde{\frac{\gamma}{2}-it}f)^{\land}\mid \Delta_{1} (O(n),F(n))$.

where $((\widetilde{\frac{\gamma}{2}+it}f)^{\land}(\phi^{r,\alpha,\lambda})$

=$\int {F(n)}f(x)e^{-ir<X_{p}^{*},X>}(\frac{\gamma}{2}(Pr_{V}(\Psi_{2}^{-1}(\overline{q_{1}(x^{-1})}))+iPr_{T}(\Psi_{2}^{-1}(\overline{q_{1}(x^{-1})})))$

$\phi_{\alpha,\lambda}(\Psi_{2}^{-1}(\overline{q_{1}(x^{-1})})dx$.

Similarly,
$((\widetilde{\frac{\gamma}{2}-it}f)^{\land}(\phi^{r,\alpha,\lambda})$

=$\int {F(n)}f(x)e^{-ir<X_{p}^{*},X>}(\frac{\gamma}{2}(Pr_{V}(\Psi_{2}^{-1}(\overline{q_{1}(x^{-1})}))-iPr_{T}(\Psi_{2}^{-1}(\overline{q_{1}(x^{-1})})))$

$\phi_{\alpha,\lambda}(\Psi_{2}^{-1}(\overline{q_{1}(x^{-1})})dx$.

Thus $(M^{+})^{l}(M^{-})^{m}G$ is the restriction of $\widehat{g}$ to $\Delta_{1} (O(n),F(n))$

where $g=(-1)^{m}(\widetilde{\frac{\gamma}{2}+it})^{l}(\widetilde{\frac{\gamma}{2}-it})^{m}f$.

Since $g\in \varphi_{O(n)}(F(n))$, it now follows that $(M^{+})^{l}(M^{-})^{m}G$ is rapidly decreasing. Thus $G\in \widehat{\varphi}(O(n),F(n))$ as desired.
The following is required to complete the Main theorem.

\begin{theorem}\label{equal}
(Inversion formula) If $f\in L_{O(n)}^{1}(F(n))\cap \in L_{O(n)}^{2}(F(n))$ is continuous, one has the Inversion Formula:

$f(x)=\frac{c}{(2\pi)^{a+2}}\int_{R}\int_{R^{\times}}\sum_{\alpha\in \wedge}(dimP_{\alpha})\widehat{f}(\phi^{r,\alpha,\lambda})\phi^{r,\alpha,\lambda}(x)\left |\lambda \right |^{a}dr d\lambda$.

=$\frac{c}{(2\pi)^{a+2}}\int_{O(n)}e^{ir<X_{p}^{*},X>}\int_{R}\int_{R^{\times}}\sum_{\alpha\in \wedge}(dimP_{\alpha})\widehat{f}(\phi^{r,\alpha,\lambda})\phi_{\alpha,\lambda}(\Psi_{2}^{-1}(\overline{q_{1}(k.x)}))\left |\lambda \right |^{a}dkdr d\lambda$.

where $c$ is a fixed constant. $x=exp(X+A)\in F(n)$.

\begin{proof}
Note that $<\phi_{\alpha}^{\circ},\phi_{\beta}^{\circ}>_{2}=\frac{(2\pi)^{a}}{d_{\alpha}}\delta_{\alpha,\beta}$.

$\int_{R^{\times}}\int_{H_{a}}\phi_{\alpha,\lambda}(z,t)\overline{\phi_{\alpha^{'},\lambda^{'}}(z,t)}dzdt\left |\lambda \right |^{a}d\lambda$

=$\int_{V}\phi_{\alpha}^{\circ}(\sqrt{\left |\lambda \right |z})\overline{\phi_{\alpha}^{\circ}(\sqrt{\left |\lambda^{'} \right |z})}dz$
$\int_{R}e^{i(\lambda-\lambda^{'})t}dt\int_{R^{\times}}\left |\lambda \right |^{a}d\lambda$

=$\int_{V}\phi_{\alpha}^{\circ}(z)\overline{\phi_{\alpha}^{\circ}(z)}dz\times 2\pi$

=$<\phi_{\alpha}^{\circ},\phi_{\alpha^{'}}^{\circ}>_{2}\times 2\pi$

=$\frac{(2\pi)^{a+1}}{d_{\alpha}}\delta_{\alpha,\alpha^{'}}$

Suppose

$g(x)=\frac{c}{(2\pi)^{a+2}}\int_{O(n)}e^{ir<X_{p}^{*},X>}\int_{R}\int_{R^{\times}}\sum_{\alpha\in \wedge}(dimP_{\alpha})\widehat{f}(\phi^{r,\alpha,\lambda})\phi^{\alpha,\lambda}(\Psi_{2}^{-1}(\overline{q_{1}(k.x)}))\left |\lambda \right |^{a}dkdr d\lambda$, then we can deduce

$\widehat{g}(\phi^{r^{'},\alpha^{'},\lambda^{'}})=\int_{F(n)}g(x)e^{-ir^{'}<X_{p}^{*},X>}\overline{\phi_{\alpha^{'},\lambda^{'}}(\Psi_{2}^{-1}(\overline{q_{1}(x)}))}dx$

=$\int_{F(n)}\int_{O(n)}\int_{R}\int_{R^{\times}}e^{ir<X_{p}^{*},X>-ir^{'}<X_{p}^{*},X>}\sum_{\alpha\in \wedge}(dimP_{\alpha})\widehat{f}(\phi^{r,\alpha,\lambda})$

$\phi_{\alpha,\lambda}(\Psi_{2}^{-1}(\overline{q_{1}(k.x)}))\overline{\phi_{\alpha^{'},\lambda^{'}}(\Psi_{2}^{-1}(\overline{q_{1}(x)}))}$
$\left |\lambda \right |^{a}dndkd\lambda dr$

Note that

$\int_{R^{\times}}\int_{F(n)}\phi_{\alpha,\lambda}(\Psi_{2}^{-1}(\overline{q_{1}(k.x)}))\overline{\phi_{\alpha^{'},\lambda^{'}}(\Psi_{2}^{-1}(\overline{q_{1}(x)}))}dx\left |\lambda \right |^{a}d\lambda=$

$\begin{cases}
\frac{1}{c}\int_{R^{\times}}\int_{H_{a}}\phi_{\alpha,\lambda}(z,t)\overline{\phi_{\alpha^{'},\lambda^{'}}(z,t)}dzdt\left |\lambda \right |^{a}d\lambda
,  & \mbox{When}k\mbox{ is e} \\
0, & \mbox{when}k\mbox{ is not equal to e}
\end{cases}$.

=$\begin{cases}
\frac{1}{c}\frac{(2\pi)^{a+1}}{d_{\alpha}}\delta_{\alpha,\alpha^{'}}\delta_{\lambda,\lambda^{'}},  & \mbox{When}k\mbox{ is e} \\
0 & \mbox{When }k\mbox{ is not equal to e}
\end{cases}$

Therefore

$\widehat{g}(\phi^{r^{'},\alpha^{'},\lambda^{'}})=\frac{c}{(2\pi)^{a+2}}\int_{R}e^{i(r-r^{'})<X_{p}^{*},X>}dr\sum_{\alpha\in \wedge}(dimP_{\alpha})\widehat{f}(\phi^{r,\alpha,\lambda^{'}})\frac{(2\pi)^{a+1}}{c}\frac{1}{d_{\alpha}}\delta_{\alpha,\alpha^{'}}$

=$\frac{c}{(2\pi)^{a+2}}\int_{R}(2\pi)(dimP_{\alpha^{'}})\widehat{f}(\phi^{r^{'},\alpha^{'}\lambda^{'}})\frac{(2\pi)^{a+1}}{c}\frac{1}{d_{\alpha^{'}}}\delta_{\alpha,\alpha^{'}}$

=$\widehat{f}(\phi^{r^{'},\alpha^{'}\lambda^{'}})$.

Therefore

$g(x)=f(x)$ for all $x\in F(n)$.
\end{proof}
\end{theorem}

\begin{lemma}\label{twist}
There exists constants $M_{1},M_{2}$ such that

$\frac{{M_{1}}}{2\left |\alpha \right |+a}\leq \frac{1}{\sum_{j=1}^{p_{1}}\delta_{j}(2\alpha_{j}+m_{j})+r^{2}}$
$\leq \frac{{M_{2}}}{2\left |\alpha \right |+a}$
\end{lemma}
\begin{proof}
Let $\delta=min_{1\leq j\leq p_{1}}\delta_{j}$. On one hand, we have

$(\sum_{j=1}^{p_{1}}\delta_{j}(2\alpha_{j}+m_{j})+r^{2})\geq (2\delta\sum_{j=1}^{p_{1}}\alpha_{j}+\delta\sum_{j=1}^{p_{1}}m_{j}+r^{2})$

$\geq (2\delta \left |\alpha \right |+\delta\sum_{j=1}^{p_{1}}m_{j}+r^{2})=\frac{1}{M_{2}}(2\delta \left |\alpha \right |+a)$

where $M_{2}=a+\delta\sum_{j=1}^{p_{1}}m_{j}+r^{2}$.

One the other hand, let $\delta^{'}=max_{1\leq j\leq p_{1}}\delta_{j}$

$\sum_{j=1}^{p_{1}}\delta_{j}(2\alpha_{j}+m_{j})+r^{2}\leq (2\delta^{'}\sum_{j=1}^{p_{1}}\alpha_{j}+\delta^{'}\sum_{j=1}^{p_{1}}m_{j}+r^{2})$

$\leq (2\sqrt{p_{1}}\delta^{'}\left |\alpha \right |+\delta^{'}\sum_{j=1}^{p_{1}}m_{j}+r^{2})$

$\leq (\sqrt{p_{1}}\delta^{'}+\delta^{'}\sum_{j=1}^{p_{1}}m_{j}+r^{2})(2\delta \left |\alpha \right |+a)$

Therefore, we can take $M_{1}=\frac{1}{\sqrt{p_{1}}\delta^{'}+\delta^{'}\sum_{j=1}^{p_{1}}m_{j}+r^{2}}$ and the inequality holds.

\end{proof}

\begin{theorem}\label{equal}
The Godement-Plamcherel measure $d\mu$ on $\Delta(O(n),F(n))$ is given by

$\int_{\Delta(O(n),F(n))}G(\psi)d\mu(\psi)=\frac{c}{(2\pi)^{a+2}}\int_{R}\int_{R^{\times}}\sum_{\alpha\in \wedge}(dimP_{\alpha})\widehat{f}(\phi^{r,\alpha,\lambda})\left |\lambda \right |^{a}dr d\lambda$.
\end{theorem}
\begin{proof}
Since existence and uniqueness of the Godement-Plancherel measure is guaranteed, we need only verify that the equation

$\int_{F(n)}\left |f(x) \right |^{2}dx=\frac{c}{(2\pi)^{a+2}}\int_{\Delta(O(n),F(n))}\left |\widehat{f}(\psi) \right |^{2}d\mu(\psi)$.

for all continuous functions $f\in L_{O(n)}^{1}(F(n))\cap L_{O(n)}^{2}(F(n))$.

Since $f(x)=\frac{c}{(2\pi)^{a+2}}\int_{R}\int_{R^{\times}}\sum_{\alpha\in \wedge}(dimP_{\alpha})\widehat{f}(\phi^{r,\alpha,\lambda})\phi^{r,\alpha,\lambda}(x)\left |\lambda \right |^{a}dr d\lambda$.

$\int_{F(n)}\left |f(x) \right |^{2}dx=\int_{F(n)}(\frac{c}{(2\pi)^{a+2}})^{2}\int_{R}\int_{R^{\times}}\int_{R}\int_{R^{\times}}\sum_{\alpha\in \wedge}(dimP_{\alpha})\widehat{f}(\phi^{r,\alpha,\lambda})\phi^{r,\alpha,\lambda}(x)\left |\lambda \right |^{a}$

$\sum_{\alpha^{'}\in \wedge}(dimP_{\alpha^{'}})\overline{\widehat{f}(\phi^{r^{'},\alpha^{'},\lambda^{'}})}\overline{\phi^{r^{'},\alpha^{'},\lambda^{'}}(x)}\left |\lambda^{'} \right |^{a}dr d\lambda dr^{'} d\lambda^{'} dx$.

Note that
$\int_{R}\int_{R^{\times}}\int_{F(n)}\phi^{r,\alpha,\lambda}(z,t)\overline{\phi^{r^{'},\alpha^{'},\lambda^{'}}(z,t)}\left |\lambda \right |^{a}drd\lambda dx$

=$\int_{R}\int_{R^{\times}}\int_{F(n)}\int_{O(n)}\int_{O(n)}e^{ir<X_{p}^{*},k.X>}\phi_{\alpha,\lambda}(\Psi_{2}^{-1}(\overline{q_{1}(k.x)}))
e^{ir^{'}<X_{p}^{*},k_{1}X>}$

$\overline{\phi_{\alpha^{'},\lambda^{'}}(\Psi_{2}^{-1}(\overline{q_{1}(k_{1}x)}))}\left |\lambda \right |^{a}dkdk_{1}drd\lambda dx$

=$\int_{R}\int_{R^{\times}}\int_{F(n)}e^{i(r-r^{'})<X_{p}^{*},X>}\phi_{\alpha,\lambda}(\Psi_{2}^{-1}(\overline{q_{1}(x)}))
\overline{\phi_{\alpha^{'},\lambda^{'}}(\Psi_{2}^{-1}(\overline{q_{1}(x)}))}\left |\lambda \right |^{a}drd\lambda dx$

=$\int_{R^{\times}}\int_{F(n)}\phi_{\alpha,\lambda}(\Psi_{2}^{-1}(\overline{q_{1}(x)}))
\overline{\phi_{\alpha^{'},\lambda^{'}}(\Psi_{2}^{-1}(\overline{q_{1}(x)}))}\left |\lambda \right |^{a}d\lambda dx\int_{R}e^{i(r-r^{'})<X_{p}^{*},X>}dr$

=$(2\pi)\delta_{r,r^{'}}\int_{R^{\times}}\int_{F(n)}\phi_{\alpha,\lambda}(\Psi_{2}^{-1}(\overline{q_{1}(x)}))
\overline{\phi_{\alpha^{'},\lambda^{'}}(\Psi_{2}^{-1}(\overline{q_{1}(x)}))}\left |\lambda \right |^{a}d\lambda dx$

=$(2\pi)\delta_{r,r^{'}}\frac{1}{c}\frac{(2\pi)^{a+1}}{d_{\alpha}}\delta_{\alpha,\alpha^{'}}\delta_{\lambda,\lambda^{'}}$

Therefore,

$\int_{F(n)}\left |f(x) \right |^{2}dx=(\frac{c}{(2\pi)^{a+2}})^{2}\int_{R}\int_{R^{\times}}\sum_{\alpha\in \wedge}(dimP_{\alpha})\widehat{f}(\phi^{r,\alpha,\lambda})\overline{\widehat{f}(\phi^{r,\alpha,\lambda})}\frac{(2\pi)^{a+2}}{(c)}\left |\lambda \right |^{a}dr d\lambda$.

$=\frac{c}{(2\pi)^{a+2}}\int_{R}\int_{R^{\times}}\sum_{\alpha\in \wedge}(dimP_{\alpha})\left |\widehat{f}(\phi^{r,\alpha,\lambda})\right |^{2} \left |\lambda \right |^{a}dr d\lambda$.

Therefore, the theorem holds.

\end{proof}

\begin{theorem}\label{equal}
Let $G$ be a bounded measurable function on $(\Delta O(n),F(n))$ correspondence to $H_{a}$
with $\left |G(r,\alpha,\lambda) \right |\leq \frac{c_{0}}{(\sum_{j=1}^{p_{1}}\delta_{j}(2\alpha_{j}+m_{j})+r^{2})^{N}}$

for some $N\geq a+3$ and some constant $c_{0}$. Then

1.$G \in L^{p}(O(n),F(n))$ for all $p\geq 1$, and

2.$G=\widehat{f}$ for some bounded continuous function $f\in L_{O(n)}^{2}(F(n))$.

Suppose, for example, that $f$ is a continuous function in $L_{O(n)}^{1}(F(n))$ with $G=\widehat{f}$ rapidly decreasing. Theorem 4.9 shows that $f$ is square integrable and that $G$ is integrable. Thus the inversion formula applies and we can cover $f$ from $G$ via

$f(x)=\frac{c}{(2\pi)^{a+2}}\int_{O(n)}e^{ir<X_{p}^{*},X>}\int_{R}\int_{R^{\times}}\sum_{\alpha\in \wedge}(dimP_{\alpha})\widehat{f}(\phi^{r,\alpha,\lambda})\phi_{\alpha,\lambda}(\Psi_{2}^{-1}(\overline{q_{1}(k.x)}))\left |\lambda \right |^{a}dkdr d\lambda$.

In particular, we see that this formula certainly holds for any function $f\in \varphi_{O(n)}(F(n))$.
\end{theorem}
\begin{proof}
To establish the first assertion, it suffices to prove that $G\in L^{1}(\delta (O(n),F(n)))$. Indeed, $\left |G(r,\alpha,\lambda) \right |^{p}$ satisfies an inequality as in the statement of the theorem with $N$ replace by $pN$. Fixed a number $K>0$, let

$A_{1}=\{\phi^{r,\alpha,\lambda}\mid 0\leq r\leq K,\left |\lambda \right |(2\left |\alpha \right |+a)\leq 1\}$,

$A_{2}=\{\phi^{r,\alpha,\lambda}\mid 0\leq r\leq K,\left |\lambda \right |(2\left |\alpha \right |+a)> 1\}$,

$A_{3}=\{\phi^{r,\alpha,\lambda}\mid r> K,\left |\lambda \right |(2\left |\alpha \right |+a)\leq 1\}$ and

$A_{4}=\{\phi^{r,\alpha,\lambda}\mid r> K,\left |\lambda \right |(2\left |\alpha \right |+a)>1\}$,

Let $M$ be a constant for which $\left |F(\psi) \right |\leq M$ for all $\psi\in \Delta (O(n),F(n))$ and let

$d_{m}=dim(P_{m}(V))=\binom{m+n-1}{m}$

Note that $\sum_{\left |\alpha \right |=m}d_{\alpha}=d_{m}\leq (m+n-1)^{n-1}$.

$\int_{A_{1}}G(\psi)d\mu(\psi)=\frac{2c}{(2\pi)^{a+2}}\int_{0}^{O(n)}\sum_{\alpha\in \wedge}(dimP_{\alpha})\int_{0<\left |\lambda \right |<\frac{1}{2\left |\alpha \right |+a}}\left |F(r,\alpha,\lambda)\right |\left |\lambda \right |^{a}dr d\lambda$.

$\leq \frac{4cMK}{(2\pi)^{a+2}}\sum_{m=0}^{\infty}d_{m}\int_{0}^{\frac{1}{2m+a}}\lambda^{a}d\lambda$

=$\frac{4cMK}{(2\pi)^{a+2}(a+1)}\sum_{m=0}^{\infty}d_{m}(\frac{1}{2m+a})^{a+1}$.

Since $d_{m}=O(m^{a-1})$, we see that the last series converges. Hence $\left |G \right |$ is integrable over $A_{1}$. Next we compute

$\int_{A_{2}}G(\psi)d\mu(\psi)=\frac{4c}{(2\pi)^{a+2}}\int_{0}^{O(n)}\sum_{\alpha\in \wedge}(dimP_{\alpha})\int_{\left |\lambda \right |>\frac{1}{2\left |\alpha \right |+a}}\left |G(r,\alpha,\lambda)\right |\left |\lambda \right |^{a}dr d\lambda$.

$\leq\frac{4cK c_{0}}{(2\pi)^{a+2}}\sum_{m=0}^{\infty}d_{m}\int_{\frac{1}{2m+a}}^{\infty}\frac{\lambda^{a}}{(\sum_{j=1}^{p_{1}}\delta_{j}(2\alpha_{j}+m_{j})+r^{2})^{N}}d\lambda$

$\leq\frac{4cKc_{0}M_{2}^{N}}{(2\pi)^{a+2}}\sum_{m=0}^{\infty}d_{m}\int_{\frac{1}{2m+a}}^{\infty}\frac{\lambda^{a}}{(2\left |\alpha \right |+a)^{N}}d\lambda$

$=\frac{4cKC_{0}M_{2}^{N}}{(2\pi)^{a+2}(N-a-1)}\sum_{m=0}^{\infty}d_{m}(\frac{1}{2m+a})^{a+1}$

The hypothesis that $N-n\geq 2$ was used above to evaluate the integral of $\frac{1}{\lambda^{N-a}}$ over $\frac{1}{2m+a}<\lambda<\infty$. Since $d_{m}=O(m^{n-1})$, we see that the series in the expansion converges. Hence $\left |F \right |$ is integrable over $A_{2}$. Similarly,

$\int_{A_{3}}G(\psi)d\mu(\psi)=\frac{2c}{(2\pi)^{a+2}}\int_{O(n)}^{\infty}\sum_{\alpha\in \wedge}(dimP_{\alpha})\int_{0<\left |\lambda \right |<\frac{1}{2\left |\alpha \right |+a}}\left |G(r,\alpha,\lambda)\right |\left |\lambda \right |^{a}dr d\lambda$.

$\leq \frac{2cc_{0}}{(2\pi)^{a+2}}\int_{K}^{\infty}\frac{1}{r^{2N}}dr\sum_{m=0}^{\infty}d_{m}\int_{0}^{\frac{1}{2m+a}}\lambda^{a}d\lambda$

$\leq \frac{4cc_{0}}{(2N-1)(2\pi)^{a+2}}\sum_{m=0}^{\infty}d_{m}\int_{0}{\frac{1}{2m+a}}\lambda^{a}d\lambda$

=$\frac{2cc_{0}}{(2N-1)(2\pi)^{a+2}(a+1)}\sum_{m=0}^{\infty}d_{m}(\frac{1}{2m+a})^{a+1}$.

Since $d_{m}=O(m^{a-1})$, we see that the last series converges. Hence $\left |F \right |$ is integrable over $A_{3}$. Finally we compute

$\int_{A_{4}}G(\psi)d\mu(\psi)=\frac{4c}{(2\pi)^{a+2}}\int_{O(n)}^{\infty}\sum_{\alpha\in \wedge}(dimP_{\alpha})\int_{\left |\lambda \right |>\frac{1}{2\left |\alpha \right |+a}}\left |G(r,\alpha,\lambda)\right |\left |\lambda \right |^{a}dr d\lambda$.

$\leq \frac{8cc_{0}M_{2}^{N-1}}{\int(2\pi)^{a+2}}\int_{O(n)}^{\infty}\frac{1}{r^{2}}\sum_{m=0}^{\infty}d_{m}\int_{\frac{1}{2m+a}}^{\infty}\frac{\lambda^{a}}{\lambda^{N-1}(2m+n)^{N-1}}
drd\lambda$.

$\leq \frac{8cc_{0}M_{2}^{N-1}}{\int(2\pi)^{a+2}(N-a-2)K}\sum_{m=0}^{\infty}d_{m}(\frac{1}{2m+a})^{a+1}$.

The hypothesis $N-a\geq 3$ was used to evaluate the integral of $\frac{1}{\lambda^{N-1-a}}$ over $\frac{1}{2m+n}<\lambda<\infty$. Since $d_{m}=O(m^{n-1})$, we see that the series in the last expression converges. Hence $\left |F \right |$ is integrable over $A_{4}$.
As $A_{1}\cup A_{2}\cup A_{3}\cup A_{4}=(\Delta_{1} (O(n),F(n)))$ is a set of full measure in $(\Delta(O(n),F(n)))$, it follows that
$F\in L^{1}(\Delta (O(n),F(n)))$.

Next let $f$ be the function on $F(n)$ defined by

$f(x)=\int_{\Delta (O(n),F(n))}G(\psi)\psi(x)d\mu(\psi)$. Since $G\in L^{1}(\Delta (O(n),F(n)))$ and the bounded spherical functions are continuous and bounded by 1, we see that $f$ is well defined, continuous and bounded by $\left \|G \right \|_{L^{1}(\Delta (O(n),F(n)))}$. Moreover, since $G\in
L^{2}(\Delta (O(n),F(n)))$ and $\land:L^{2}(\Delta (O(n),F(n)))\rightarrow L_{O(n)}^{2}(F(n))$ is an isometry, we have that $f\in L_{O(n)}^{2}(F(n))$
with $\left \|f \right \|_{2}^{2}=\int_{\Delta (O(n),F(n))}\left |G(\psi) \right |^{2}d\mu(\psi)$ and $\widehat{f}=G$. This establish the second assertion in this Theorem.
\end{proof}

Remark. One can show that the set $A_{1}$ used in the proof of the above theorem is compact in $\Delta (O(n),F(n))$. Thus any bounded measurable function is integrable over $A_{1}$. This observation motivates the decomposition used in the proof.

Proof that $\widehat{\varphi}(O(n),F(n))\subseteq \varphi_{O(n)}(F(n))^{\land}$. Suppose now that $G\in \widehat{\varphi}(O(n),F(n))$. Above theorem shows that $G=\widehat{f}$ where

$f(x)=\frac{c}{(2\pi)^{a+2}}\int_{R}\int_{R^{\times}}\sum_{\alpha\in \wedge}(dimP_{\alpha})\widehat{f}(\phi^{r,\alpha,\lambda})\phi^{r,\alpha,\lambda}(x)\left |\lambda \right |^{a}dr d\lambda$.

=$\frac{c}{(2\pi)^{a+2}}\int_{O(n)}e^{ir<X_{p}^{*},X>}\int_{R}\int_{R^{\times}}\sum_{\alpha\in \wedge}(dimP_{\alpha})\widehat{f}(\phi^{r,\alpha,\lambda})\phi_{\alpha,\lambda}(\Psi_{2}^{-1}(\overline{q_{1}(k.x)}))\left |\lambda \right |^{a}dkdr d\lambda$.
is $O(n)$-invariant, continuous, bounded and square integrable. To show that $f\in \varphi_{O(n)}^{F(n)}$, I will show that $f$ is smooth and that

$(\widetilde{\frac{\gamma^{2}}{4}+t^{2}})^{a}(\widetilde{\frac{\partial}{\partial t}})^{b}\widetilde{\Delta}^{c}f\in L_{O(n)}^{2}(F(n))$

where $\widetilde{\frac{\gamma^{2}}{4}+t^{2}}, \widetilde{\frac{\partial}{\partial t}}, \widetilde{\Delta}$ will be defined below.

for all $a,b,c\geq 0$. This will follows from the facts

1.$\widetilde{\Delta}^{c}f\in L_{O(n)}^{2}(F(n))$ with $(\widetilde{\Delta}f)^{\land}\in \widehat{\varphi}(O(n),F(n))$.

2.$\widetilde{\frac{\partial}{\partial t}}f\in L_{O(n)}^{2}(F(n))$ with $(\widetilde{\frac{\partial}{\partial t}}f)^{\land}\in \widehat{\varphi}(O(n),F(n))$.

3.$(\widetilde{\frac{\gamma}{2}\pm it}f)\in L_{O(n)}^{2}(F(n))$ with $(\widetilde{\frac{\gamma}{2}\pm it}f)^{\land}\in \widehat{\varphi}(O(n),F(n))$.

which I will prove below. Here $\Delta=\sum_{j=1}^{a}\frac{\partial}{\partial z_{j}}.\frac{\partial}{\partial \overline{z_{j}}}$ is in
formal.

Using equation $\widehat{U}(\phi_{\alpha,\lambda})=-\left |\lambda \right |(2\left |\alpha \right |+a)$ for the eigenvalues of the Heisenberg sub-Laplacian, one obtains

$-\left |\lambda \right |(2\left |\alpha \right |+a)\phi_{\alpha,\lambda}=U\phi_{\alpha,\lambda}$
=$[4\Delta-\frac{\gamma}{2}(\frac{\partial}{\partial t})^{2}]\phi_{\alpha,\lambda}$

=$4\Delta \phi_{\alpha,\lambda}-\frac{\lambda^{2}}{2}\gamma \phi_{\alpha,\lambda}$

Since $\gamma \phi_{\alpha,\lambda}=-\frac{1}{\left |\lambda \right |}(D^{+}-D^{-})\phi_{\alpha,\lambda}$. Therefore,

$4\Delta \phi_{\alpha,\lambda}=-\frac{\left |\lambda \right |}{2}(D^{+}-D^{-})\phi_{\alpha,\lambda}-\left |\lambda \right |(2\left |\alpha \right |+a)\phi_{\alpha,\lambda}$

=$-\frac{\left |\lambda \right |}{2}[\sum_{\left |\beta \right |=\left |\alpha \right |+1}\frac{d_{\beta}}{d_{\alpha}}\begin{bmatrix}
\beta\\
\alpha
\end{bmatrix}\phi_{\beta,\lambda}+(2\left |\alpha \right |+a)\phi_{\alpha,\lambda}$

$+\sum_{\left |\beta \right |=\left |\alpha \right |-1}\begin{bmatrix}
\alpha\\
\beta
\end{bmatrix}\phi_{\beta,\lambda}]$

Define a function $G_{\Delta}$ on $\Delta_{1} (O(n),F(n))$ by

$G_{\Delta}(r,\alpha,\lambda)=-\frac{\left |\lambda \right |}{2}(D^{+}-D^{-})G(r,\alpha ,\lambda)-\left |\lambda \right |\times (2 \left |\alpha \right |+a)G(r,\alpha ,\lambda)$

=$-\frac{\left |\lambda \right |}{2}[\sum_{\left |\beta \right |=\left |\alpha \right |+1}\frac{d_{\beta}}{d_{\alpha}}\begin{bmatrix}
\beta\\
\alpha
\end{bmatrix}G(r,\beta,\lambda)+(2\left |\alpha \right |+a)F(r,\alpha,\lambda)$

$+\sum_{\left |\beta \right |=\left |\alpha \right |-1}\begin{bmatrix}
\alpha\\
\beta
\end{bmatrix}G(r,\beta,\lambda)]$

It is not hard to show that $G_{\Delta}\in \widehat{\varphi}(O(n),F(n))$.

In particular, note the equations

$\sum_{\left |\beta \right |=\left |\alpha \right |-1}\begin{bmatrix}
\alpha\\
\beta
\end{bmatrix}=\left |\alpha \right |$

$\sum_{\left |\beta \right |=\left |\alpha \right |+1}\frac{d_{\beta}}{d_{\alpha}}\begin{bmatrix}
\beta\\
\alpha
\end{bmatrix}=\left |\alpha \right |+a$.
 and

$\frac{{M_{1}}}{2\left |\alpha \right |+a}\leq \frac{1}{\sum_{j=1}^{p_{1}}\delta_{j}(2\alpha_{j}+m_{j})+r^{2}}$
$\leq \frac{{M_{2}}}{2\left |\alpha \right |+a}$
for some fixed $M_{1}>0$ and $M_{2}>0$ give

$\left |G_{\Delta}(r,\alpha,\lambda) \right |$

$\leq \frac{\left |\lambda \right |}{2}[(\left |\alpha \right |+a)\sum_{\left |\beta \right |=\left |\alpha \right |+1}\left |G(r,\beta,\lambda)\right |+(2\left |\alpha \right |+a)\left |G(r,\alpha,\lambda)\right |$

$+\left |\alpha \right |\sum_{\left |\beta \right |=\left |\alpha \right |-1}\left |G(r,\beta,\lambda)\right |]$

One uses this to show that $G_{\Delta}$ satisfies estimates as in Definition 4.1. Moreover, Lemma 3.9 shows that for each $\lambda\neq 0$,

$\sum_{\alpha\in \wedge}d_{\alpha}G(r,\alpha,\lambda)(D^{+}-D^{-})\phi_{\alpha,\lambda}=$

$\sum_{\alpha\in \wedge}d_{\alpha}(D^{+}-D^{-})G(r,\alpha,\lambda)\phi_{\alpha,\lambda}$.

and hence also

$\frac{c}{(2\pi)^{a+2}}\int_{O(n)}e^{ir<X_{p}^{*},X>}\int_{R}\int_{R^{\times}}\sum_{\alpha\in \wedge}(dimP_{\alpha})G_{\Delta}(\phi^{r,\alpha,\lambda})\phi_{\alpha,\lambda}(\Psi_{2}^{-1}(\overline{q_{1}(k.x)}))\left |\lambda \right |^{a}dkdr d\lambda$

=$\frac{4c}{(2\pi)^{a+2}}\int_{O(n)}e^{ir<X_{p}^{*},X>}\int_{R}\int_{R^{\times}}\sum_{\alpha\in \wedge}(dimP_{\alpha})G(\phi^{r,\alpha,\lambda})\Delta \phi_{\alpha,\lambda}(\Psi_{2}^{-1}(\overline{q_{1}(k.x)}))\left |\lambda \right |^{a}dkdr d\lambda$

=$4\widetilde{\Delta}f$.

We conclude $\widetilde{\Delta}f\in L_{O(n)}^{2}(F(n))$ with $4(\Delta f)^{\land}=G_{\Delta}\in \widehat{\varphi}(O(n),F(n))$. This proves item 1 above.

Next note that the function defined on $\Delta_{1} (O(n),F(n))$ by $\lambda G(r,\alpha,\lambda)$ belongs to $\widehat{\varphi}(O(n),F(n))$.
Since $\frac{\partial \phi_{\alpha,\lambda}}{\partial t}=i\lambda \phi_{\alpha,\lambda}$, we see that

$\widetilde{\frac{\partial f}{\partial t}}(X)=\frac{c}{(2\pi)^{a+2}}\int_{O(n)}e^{ir<X_{p}^{*},X>}\int_{R}\int_{R^{\times}}\sum_{\alpha\in \wedge}(dimP_{\alpha})\widehat{f}(\phi^{r,\alpha,\lambda})\frac{\partial}{\partial t}\phi_{\alpha,\lambda}(\Psi_{2}^{-1}(\overline{q_{1}(k.x)}))\left |\lambda \right |^{a}dkdr d\lambda$

$=i\lambda f(x)$.
Therefore, $\widetilde{\frac{\partial f}{\partial t}}\in L_{O(n)}^{2}(F(n))$ with $(\widetilde{\frac{\partial f}{\partial t}})^{\land}=i\lambda G\in \widehat{\varphi}(O(n),F(n))$. This establish 2 above.

We begin the proof of item 3 by setting

$\widetilde{G}(r,a,\lambda,k)=\sum_{\alpha\in \wedge}(dimP_{\alpha})G(\phi^{r,\alpha,\lambda})\phi_{\alpha,\lambda}^{0}(Pr_{V}\Psi_{2}^{-1}(\overline{q_{1}(k.x)}))\left |\lambda \right |^{a}$ for each $\lambda \neq 0$, so that

$f(x)=\frac{c}{(2\pi)^{a+2}}\int_{O(n)}e^{ir<X_{p}^{*},X>}\int_{R}\int_{R^{\times}}\widetilde{G}(r,a,\lambda,k)e^{irPr_{V}\Psi_{2}^{-1}(\overline{q_{1}(k.x)})}dkdr d\lambda$.

We denote $Pr_{V}\Psi_{2}^{-1}(\overline{q_{1}(k.x)})=Z_{k.n}\in V$.

Note that we can compute $\partial_{\lambda}\widetilde{G}$ by taking derivatives term-wise in above equation. For $\lambda>0$, we have

$\partial_{\lambda}\widetilde{G}(r,a,\lambda,k)=\sum_{\alpha\in \wedge}(dimP_{\alpha})d_{\alpha}a\lambda^{a-1} G(\phi^{r,\alpha,\lambda})\phi_{\alpha,\lambda}^{0}(Z_{k.n})$

$+\sum_{\alpha\in \wedge}(dimP_{\alpha})d_{\alpha}\lambda^{a} \partial_{\lambda}G(\phi^{r,\alpha,\lambda})\phi_{\alpha,\lambda}^{0}(Z_{k.n})$

$+\sum_{\alpha\in \wedge}(dimP_{\alpha})d_{\alpha}\lambda^{a} G(\phi^{r,\alpha,\lambda})\partial_{\lambda}\phi_{\alpha,\lambda}^{0}(Z_{k.n})$

Since $G\in \widehat{\varphi}(O(n),F(n))$, the estimates in Definition 4.1 can be applied to show that the first two sum converges absolutely for $\lambda>0$. For the third sum, I use Equation 3.3 for $\partial_{\lambda}\phi_{\alpha,\lambda}^{0}(Z)=\partial_{\lambda}\phi_{\alpha,\lambda}(z,0)$. together with the lemma 3.9 to derive two identities.

$\sum_{\alpha\in \wedge}(dimP_{\alpha})G(\phi^{r,\alpha,\lambda})\phi_{\alpha,\lambda}^{0}(Z_{k.n})\lambda^{a}$

=$\begin{cases}
-\frac{\gamma(Z_{k.n})}{2}\sum_{\alpha\in \wedge}(dimP_{\alpha})\lambda^{a}G(\phi^{r,\alpha,\lambda})\phi_{\alpha,\lambda}^{0}(Z_{k.n})+\sum_{\alpha\in \wedge}(dimP_{\alpha})\lambda^{a}G(\phi^{r,\alpha,\lambda})\frac{1}{\lambda}D^{-}\phi_{\alpha,\lambda}^{0}(Z_{k.n}) \\
\frac{\gamma(Z_{k.n})}{2}\sum_{\alpha\in \wedge}(dimP_{\alpha})\lambda^{a}G(\phi^{r,\alpha,\lambda})\phi_{\alpha,\lambda}^{0}(Z_{k.n})+\sum_{\alpha\in \wedge}(dimP_{\alpha})\lambda^{a}G(\phi^{r,\alpha,\lambda})\frac{1}{\lambda}D^{+}\phi_{\alpha,\lambda}^{0}(Z_{k.n}) \\
\end{cases}$

=$\begin{cases}
-\frac{\gamma(Z_{k.n})}{2}\widetilde{G}(r,a,\lambda,k)-\sum_{\alpha\in \wedge}(dimP_{\alpha})\lambda^{a-1}(D^{+}+a)G(\phi^{r,\alpha,\lambda})\phi_{\alpha,\lambda}^{0}(Z_{k.n}) \\
\frac{\gamma(Z_{k.n})}{2}\widetilde{G}(r,a,\lambda,k)-\sum_{\alpha\in \wedge}(dimP_{\alpha})\lambda^{a-1}(D^{-}+a)G(\phi^{r,\alpha,\lambda})\phi_{\alpha,\lambda}^{0}(Z_{k.n}) \\
\end{cases}$

Substituting these identities in the expansion for
\begin{equation}\label{l-invariant elements}
\begin{split}
\partial_{\lambda}\widetilde{G}(r,a,\lambda,k)=
\begin{cases}
-\frac{\gamma(Z_{k.n})}{2}\widetilde{G}(r,a,\lambda,k)-\sum_{\alpha\in \wedge}(dimP_{\alpha})\lambda^{a}(\partial_{\lambda}-\frac{1}{\lambda}D^{+})G(\phi^{r,\alpha,\lambda})\phi_{\alpha,\lambda}^{0}(Z_{k.n}) \\
\frac{\gamma(Z_{k.n})}{2}\widetilde{G}(r,a,\lambda,k)-\sum_{\alpha\in \wedge}(dimP_{\alpha})\lambda^{a}(\partial_{\lambda}-\frac{1}{\lambda}D^{-})G(\phi^{r,\alpha,\lambda})\phi_{\alpha,\lambda}^{0}(Z_{k.n}) \\
\end{cases}
\end{split}
\end{equation}
both valid for $\lambda>0$. We have similar identities for $\lambda<0$:
\begin{equation}\label{l-invariant elements}
\begin{split}
\partial_{\lambda}\widetilde{G}(r,a,\lambda,k)=
\begin{cases}
-\frac{\gamma(Z_{k.n})}{2}\widetilde{G}(r,a,\lambda,k)-\sum_{\alpha\in \wedge}(dimP_{\alpha})\lambda^{a}(\partial_{\lambda}-\frac{1}{\lambda}D^{-})G(\phi^{r,\alpha,\lambda})\phi_{\alpha,\lambda}^{0}(Z_{k.n}) \\
\frac{\gamma(Z_{k.n})}{2}\widetilde{G}(r,a,\lambda,k)-\sum_{\alpha\in \wedge}(dimP_{\alpha})\lambda^{a}(\partial_{\lambda}-\frac{1}{\lambda}D^{+})G(\phi^{r,\alpha,\lambda})\phi_{\alpha,\lambda}^{0}(Z_{k.n}) \\
\end{cases}
\end{split}
\end{equation}
Note that $(\partial_{\lambda}-\frac{1}{\lambda}D^{\pm})G$ is the restriction of $M^{\pm}G$ to

$\Delta_{1}^{+} (O(n),F(n))=\{\phi^{r,\alpha,\lambda}\mid r\in R, \alpha\in \wedge, \lambda>0\}$ and also of $M^{\mp}G$ to

$\Delta_{1}^{-} (O(n),F(n))=\{\phi^{r,\alpha,\lambda}\mid r\in R, \alpha\in \wedge, \lambda>0\}$. Since $M^{\pm}G\in \widehat{\varphi}(O(n),F(n))$, $M^{\mp}G$ is integrable on $\Delta (O(n),F(n))$ and equation 4.10, 4.11 show that $\partial_{\lambda}\widetilde{G}(r,a,\lambda,k)$ is integrable on $R^{\times}=\{\lambda \mid \lambda \neq 0\}$. We have

$\frac{(2\pi)^{a+2}}{c}\widetilde{it}f(x)=\int_{O(n)}e^{ir<X_{p}^{*},X>}\int_{R}\int_{R^{\times}}\widetilde{G}(r,a,\lambda,k)\partial_{\lambda}(e^{irPr_{V}\Psi_{2}^{-1}(\overline{q_{1}(k.x)})})dkdr d\lambda$.

=$\int_{O(n)}e^{ir<X_{p}^{*},X>}\int_{R}(\int_{0}^{\infty}\widetilde{G}(r,a,\lambda,k)\partial_{\lambda}(e^{irPr_{V}\Psi_{2}^{-1}(\overline{q_{1}(k.x)})})d\lambda$

$+\int_{-\infty}^{0}\widetilde{G}(r,a,\lambda,k)\partial_{\lambda}(e^{irPr_{V}\Psi_{2}^{-1}(\overline{q_{1}(k.x)})}))d\lambda
dkdr$.

=$\int_{O(n)}e^{ir<X_{p}^{*},X>}\int_{R}(-\int_{0}^{\infty}\partial_{\lambda}\widetilde{G}(r,a,\lambda,k)e^{irPr_{V}\Psi_{2}^{-1}(\overline{q_{1}(k.x)})}d\lambda$

$-\int_{-\infty}^{0}\partial_{\lambda}\widetilde{G}(r,a,\lambda,k)e^{irPr_{V}\Psi_{2}^{-1}(\overline{q_{1}(k.x)})})d\lambda$.

$-\lim_{\lambda\rightarrow 0^{+}}\widetilde{G}(r,a,\lambda,k)+\lim_{\lambda\rightarrow 0^{-}}\widetilde{G}(r,a,\lambda,k))dkdr$

It can be shown that the limits $\lim_{\lambda\rightarrow 0^{\pm}}\widetilde{G}(r,a,\lambda,k)$ exists and are equal. Here one need to use the hypothesis that $F$ is continuous across $\Delta_{2} (O(n),F(n))$ and that $G_{0}$ is a Schwartz function. Using the last two equations for
$\partial_{\lambda}\widetilde{G}(r,a,\lambda,k)$ we obtain

$\frac{(2\pi)^{a+2}}{c}(\widetilde{\pm \frac{\gamma}{2}\pm it})f(x)=\int_{O(n)}e^{ir<X_{p}^{*},X>}\int_{R}(-\int_{0}^{\infty}\partial_{\lambda}\widetilde{G}(r,a,\lambda,k)e^{irPr_{V}\Psi_{2}^{-1}(\overline{q_{1}(k.x)})}d\lambda$

$-\int_{-\infty}^{0}\partial_{\lambda}\widetilde{G}(r,a,\lambda,k)e^{irPr_{V}\Psi_{2}^{-1}(\overline{q_{1}(k.x)})}d\lambda)drdk$.

$=\frac{c}{(2\pi)^{a+2}}\int_{R}\int_{R^{\times}}\sum_{\alpha\in \wedge}(dimP_{\alpha})M^{\pm}G(\phi^{r,\alpha,\lambda})\phi^{r,\alpha,\lambda}(x)\left |\lambda \right |^{a}dr d\lambda$.

We conclude that $(\widetilde{\pm \frac{\gamma}{2}\pm it})f\in \in L_{O(n)}^{2}(F(n))$ with $\frac{(2\pi)^{a+2}}{c}(\widetilde{\pm \frac{\gamma}{2}\pm it}f)^{\land}=M^{\pm}G\in \widehat{\varphi}(O(n),F(n))$. This completes the proof of item 3.

\newpage
\bibliographystyle{ieeetr}
\addcontentsline{toc}{chapter}{Bibliography}

\end{document}